\documentclass[11pt]{article}

\usepackage[latin1]{inputenc}
\usepackage{epsfig}
\usepackage{latexsym,amssymb,amsmath}
\usepackage{amsthm}
\newcommand{\bfi}{\bfseries\itshape}

\newtheorem{thm}{Theorem}[section]
\newtheorem{prop}[thm]{Proposition}
\newtheorem{lem}[thm]{Lemma}

\newtheorem{dfn}[thm]{Definition}

\renewcommand\footnotemark{}

\begin{document}

\title{Lagrange-d'Alembert-Poincar\'{e} equations by several stages}

\author{Hern{\'a}n Cendra* and Viviana A. D\'{\i}az**
\\ Departamento de Matem\'atica
\\ Universidad Nacional del Sur, Av. Alem 1253
\\ 8000 Bah\'{\i}a Blanca and CONICET, Argentina
\\{*\footnotesize hcendra@gmail.com}
\\{**\footnotesize viviana.diaz@uns.edu.ar}
\thanks{The authors wish to thank the following institutions for making our work on this article possible: Universidad Nacional del Sur (projects PGI 24/L075 and PGI 24/ZL06); Agencia Nacional de Promoci\'on Cient\'{i}fica y Tecnol\'{o}gica, Argentina (projects PICT 2006-2219 and PICT 2010-2746); CONICET, Argentina (project PIP 2010-2012 11220090101018); European Community, FP7 (project IRSES ``GEOMECH'' 246981).}}

\date{\today}


\maketitle


\begin{abstract}
The aim of this paper is to write explicit expression in terms of a given principal connection of the Lagrange-d'Alembert-Poincar\'{e} equations in several stages. This is obtained by using a reduced Lagrange-d'Alembert's Principle in several stages, extending methods introduced for the case of two stages by one of the authors and collaborators. The case of the Euler's disk is described as an illustrative example.
\end{abstract}
\tableofcontents

\section{Introduction}

The topic of nonholonomic mechanical systems is an old one and has been very active for the last few decades. It is related to applications, like robotics, locomotion, phases and others and also to questions of a purely mathematical interest.

 The equations go\-ver\-ning nonholonomic systems are Lagrange-d'Alembert Equations which are derived by the Lagrange-d'Alembert Principle.
Reduction of Euler-Lagrange and Lagrange-d'Alembert Equations by a group of symmetries is a fundamental issue in mechanics. The basic facts on reduction by one stage are reviewed in detail in Appendix \ref{appendixb}.\\

This paper  is devoted to a specific question, namely, to study the extension of the reduction theory by one stage developed in  \cite{CMR01b,CMR01a}, for the case of several stages. In particular, one obtains explicit expressions of
Lagrange-d'Alembert-Poincar\'e Equations by $n$ stages, written in terms of a given principal connection. As in \cite{CMR01b,CMR01a}, we will use the idea of reducing the variational principle  rather than reducing the equations, including a detailed study of the geometry of variations.
The question of
writing explicit equations for the case of several stages,
where one has a chain of normal subgroups of the
given symmetry group $G$ such that
$G = N_0 \vartriangleright N_1 \vartriangleright ...\vartriangleright N_{n + 1} = \{e\}$ rather than just none or one,
is clearly natural to complete the theory and deal with a wider class of examples.
We describe the illustrative example of Euler's Disk, obtaining reduced equations by two stages. These equations are the ones previously obtained in \cite{CDgrueso}, which were solved by using hypergeometric functions, which means some kind of integrability.\\

The literature on nonholonomic systems and applications is overwhelming. Here are some references that we think one should keep in mind while reading the present paper. General references on mechanics and nonholonomic systems: \cite{foundation,metodosmatem,carira93,CMR01b,CMR01a,dLdD96b,EKMR05,Rui02,HamiltonianReductionAmarillo,
marsden3,neimarkfufaev,marle98,vershik84,vershikfaddeev72}; applications of nonholonomic systems: \cite{bloch,CDgrueso,koonmarsden97c,marle95,HamiltonianReductionAmarillo,neimarkfufaev,marle98}; nonholonomic systems with symmetry and reduction: \cite{batesGmD96,batessniatycki,BKMM96,CdLMdD99,CMR01b,CMR01a,CMRY09,marle95,
HamiltonianReductionAmarillo,marle98}.

The paper is reasonably self-contained and the results that constitute the basic background are clearly stated in the appendices A and B.

In section \ref{sect:corcheteetapas} we describe one of the main ingredients for the paper, namely, an explicit formula for the Lie bracket by several stages.
In section \ref{sect:Euler-d'Alem-Poin} we study an important particular case of the general theory: the Euler-d'Alembert-Poincar\'e equations by stages.
In sections \ref{sect:Lag-Poin} and \ref{sect:Lag-d'Alem-Poin} we obtain the main results of the paper, that is, the Lagrange-Poincar\'e and the Lagrange-d'Alembert-Poincar\'e equations by several stages.
In section \ref{sect:disco} we describe the example of Euler's disk.

Some future studies are in order, like
establishing the connection with Hamiltonian reduction by
stages \cite{HamiltonianReductionAmarillo} and, more generally, Dirac-Weinstein reduction by
stages \cite{CMRY09}.\\


\section{The Lie bracket by stages}\label{sect:corcheteetapas}
In this section we calculate explicit expressions for the Lie bracket by several stages. For doing this,
we need some background on reduced covariant derivatives and associated connections on  associated bundles, which we recall in Appendix \ref{appendixa}.
The formulas that we obtain are fundamental for writing reduced equations of motion by several stages, which is the main purpose of the paper.

\subsection{An explicit formula for the Lie bracket by two stages} \label{sec:LieBracket}

In this section we will use results and notation from  \cite{CMR01a}, see also
Appendix \ref{appendixa}.
\begin{large}\end{large}
Let a given principal bundle $\pi:Q\rightarrow Q/G$
with structure group $G,$ and choose a Riemannian invariant metric on $Q$.
Let
$N$ be a normal subgroup of $G$,  then for each $q\in Q$ we have a decomposition of $TQ$ as an orthogonal direct
sum $T_qQ=\text{Ver}^N(T_qQ)\oplus H_N(q),$ where $H_N(q)$ is the
orthogonal complement of $Ver^N(T_qQ)$.
Then, this
collection of $H_N(q)$ defines a connection on the principal bundle
$Q$ with structure group $N.$
Let $\mathcal{A}_N$ be the
corresponding 1-form connection.
From the definition of $\mathcal{A}_N$ it follows easily that for each $g\in G$ and every
$v_q\in T_qQ$ we have that $\mathcal{A}_N(g v_q)=Ad_g
\mathcal{A}_N(v_q).$

\vspace{.3cm}
Next, we use the notation and results on quotient horizontal and quotient vertical connections, summarized in Appendix \ref{appendixa}.
Given $[q_0,\xi_0]_G\in\widetilde{\mathfrak{g}},$
 let $X_0$ be the $G$-invariant vertical vector field defined by
$X_0(q_0)=\xi_0 q_0,$
which we can prove that it is well defined, and satisfies
$\mathcal{A}\left( X_0 (q_0)\right) = \xi_0.$

Let $W \rightarrow Q$ be a vector bundle where $G$ acts by vector bundle isomorphisms and let
$w \in \Gamma^G (W),$
which is identified with
$[w]_G$.
We define the
\textbf{\textit{quotient,
or reduced, vertical connection}}
$[\nabla^{(A,V)} ]_{G,[q_0,\xi_0]_G}[w]_G$
by
\begin{equation}\label{conexionverticalcociente}
\left[\nabla^{(\mathcal{A},V)}\right]_
{G,[q_0,\xi_0]_G}[w]_G:=\left[\nabla_{X_0}w\right]_G,
\end{equation}

which, in the present context, is equivalent to the one given in \cite{CMR01a}
(see equation \ref{dfn5120} in Appendix \ref{appendixa}).

\vspace{.5cm}
The following proposition has been proven in \cite{CMR01a}, Corollary 6.3.11.
\begin{prop} \label{propcorchgral}Consider a
Lie group
$G,$ $N$
a normal subgroup of
$G$ and
$K=G/N.$
We sometimes think of
$G$ as a principal bundle
$\pi_N : G \rightarrow K$ with structure group $N$  acting on the left, and also,
as a principal bundle
$G \rightarrow G/G\equiv\{[e_G]_G\},$ where $e_G$ is the neutral element of $G,$ the base is a single point and the structure group $G$ acts on the left by left translations.
 Similarly, we consider sometimes $K$
as a principal bundle where the base is a single point and the structure group $K$ acts on the left by left translations,
$K \rightarrow K/K\equiv\{[e_K]_K\},$ where $e_K$ is the neutral element of $K.$
Let
$\mathfrak{g},\ \mathfrak{n}$
and $\mathfrak{k}$ be the Lie algebras of
$G,\ N$ and $K$ respectively.
We call
$ e_N = e_G$
the neutral element of
$N.$ The adjoint bundle $\tilde{\mathfrak{g}} \rightarrow \{[e_G]_G\}$
of
$G \rightarrow \{[e_G]_G\},$
is naturally identified with
$\mathfrak{g},$
say
$\mathfrak{g} \equiv \tilde{\mathfrak{g}},$
by
$\xi \equiv [e_G, \xi]_G.$  Likewise  there is a natural identification
$\mathfrak{k} \equiv \tilde{\mathfrak{k}}$
given by
$\kappa \equiv [e_G, \kappa]_G.$

Let an arbitrary identification
$\mathfrak{g}\equiv \mathfrak{k}\oplus\mathfrak{n}$ as linear spaces. Choose a $G$-invariant Riemannian metric on $G,$ which is determined by its restriction to
$\mathfrak{g}.$
Let  $\mathcal{A}_N$ be the principal connection defined on the
principal bundle $G \rightarrow K$ with structure group $N$ in the way described at the beginning of this section, which has the property that $\mathcal{A}_N(g v_q)=\text{Ad}_g
\mathcal{A}_N(v_q),$ for every $g,q\in G,\ v_q\in T_qG.$
 Then
$\mathfrak{g}=\mathfrak{k}^{\mathcal{A}_N}\oplus\mathfrak{n}$ where
$\mathfrak{k}^{\mathcal{A}_N}$ is the horizontal lift of
$\mathfrak{k} = T_{e_K}K$ in the bundle $G\rightarrow K,$
at the point
$e_G \in G.$ Note that, by definition,
the subspaces
$\mathfrak{k}^{\mathcal{A}_N}$ and $\mathfrak{n}$ of $\mathfrak{g}$
are orthogonal.

Then, for the given identification
$\mathfrak{g}\equiv \mathfrak{k}\oplus\mathfrak{n}$, the bracket
on the Lie algebra $\mathfrak{g}$ can be written in terms of
the brackets on the Lie algebra $\mathfrak{n}$ and the Lie algebra
$\mathfrak{k},$ and also in terms of $\nabla^{(\mathcal{A}_N,V)}$ and
$\widetilde{B}^{\mathcal{A}_N}$ using the formula, given in \cite{CMR01a},

\begin{equation}\label{corcheteCMR}
\left[\kappa_1\oplus\eta_1,\kappa_2\oplus\eta_2\right]=
\left[\kappa_1,\kappa_2\right]\oplus[\nabla^{(\mathcal{A}_N,V)}]_{K,\kappa_1}\eta_2-
[\nabla^{(\mathcal{A}_N,V)}]_{K,\kappa_2}\eta_1-
[\widetilde{B}^{\mathcal{A}_N}]_{K}(\kappa_1,\kappa_2)+
\left[\eta_1,\eta_2\right].
\end{equation}

\vspace{.2cm}
Formula (\ref{corcheteCMR}) should be interpreted as follows.
By definition $\widetilde{\mathfrak{n}}\rightarrow K$ is a vector bundle on $K$ which is a left principal bundle  $K\rightarrow\{[e_K]_K\}$ with structure group $K$ over a single point.
We have an action of $K$ on $\widetilde{\mathfrak{n}}$ covering the action of $K$ on $K\rightarrow\{[e_K]_K\}$ defined by $[g]_N[h,\eta]_N=[gh,Ad_g\eta]_N.$

For each $\eta\in\mathfrak{n}$ there is a $K$-invariant section $\sigma_{\eta}:K\rightarrow \widetilde{\mathfrak{n}}$ defined by $\sigma_{\eta}([g]_N)=[g,Ad_g\eta]_N.$
So, we can identify
\begin{equation}\label{qqqq}
 \sigma_{\eta}\equiv[\sigma_{\eta}]_K\equiv\eta.
\end{equation}

Then the right hand side of (\ref{corcheteCMR}) should be interpreted as being
\begin{align} \nonumber
\left[\kappa_1\oplus\eta_1,\kappa_2\oplus\eta_2\right]
&=
\left[\kappa_1,\kappa_2\right]\oplus[\nabla^{(\mathcal{A}_N,V)}]_{K,[e_K,\kappa_1]_K}
[\sigma_{\eta_2}]_K -
[\nabla^{(\mathcal{A}_N,V)}]_{K,[e_K,\kappa_2]_K}[\sigma_{\eta_1}]_K \\
&-
[\widetilde{B}^{\mathcal{A}_N}]_{K}(\kappa_1,\kappa_2)+\left[[e_K,\eta_1]_K,[e_K,\eta_2]_K\right]_N,
\end{align}

using the identifications
$\displaystyle\left[[\sigma_{\eta_1}]_K,[\sigma_{\eta_2}]_K\right]\equiv
\left[[\sigma_{\eta_1},\sigma_{\eta_2}]\right]_K\equiv[\sigma_{[\eta_1,\eta_2]}]_K
\equiv[\eta_1,\eta_2].$
Now, according to the general definition given in \cite{CMR01a} that we recall in formula (\ref{betilde}) in Appendix \ref{appendixb}, $\widetilde{B}^{\mathcal{A}_N}$ is a $\widetilde{\mathfrak{n}}$-valued 2-form on $K$ given by
$\widetilde{B}^{\mathcal{A}_N}(k)(\dot{k},\delta k)=[g,B(g)(\dot{g},\delta g)]_N,$ where $k=[g]_N,$ $\dot{k}=T\pi_N\dot{g}$ and $\delta k=T\pi_N \delta g.$ It can be checked that $\widetilde{B}^{\mathcal{A}_N}$ is $K$-invariant. Then
\begin{align*}
[\widetilde{B}^{\mathcal{A}_N}]_K
(\kappa_1,\kappa_2)
&: =
[\widetilde{B}^{\mathcal{A}_N}(e_K)(\kappa_1,\kappa_2)]_K\\
&=
\left[[e_G, B^{\mathcal{A}_N}(e_G)(\kappa_1^{\mathcal{A}_N}(e_G),
\kappa_2^{\mathcal{A}_N}(e_G))]_N\right]_K \\
&\equiv
B^{\mathcal{A}_N}(e_G)(\kappa_1^{\mathcal{A}_N}(e_G),
\kappa_2^{\mathcal{A}_N}(e_G)),
\end{align*}
being
$\kappa_i^{\mathcal{A}_N}(e_G)\in\mathfrak{g}$ the horizontal lift of
$\kappa_i$ at $e_G$ to $G$ for $i=1,2.$
\end{prop}

In order to obtain a more explicit expression of the Lie bracket
$\left[\kappa_1\oplus\eta_1,\kappa_2\oplus\eta_2\right]$
in terms of the decomposition $\mathfrak{g}=\mathfrak{k}\oplus\mathfrak{n},$ we shall see explicitly how to compute
$[\widetilde{B}^{\mathcal{A}_N}]_{G/N}(\kappa_1,\kappa_2)$ and
$[\nabla^{(\mathcal{A}_N,V)}]_{G/N,\kappa}\eta$
and identify them with elements of $\mathfrak{n},$
using the decomposition $\mathfrak{g}=\mathfrak{k}\oplus\mathfrak{n}$ and the identifications $\eta\leftrightarrow\sigma_{\eta}$ and $\kappa\leftrightarrow[e_K,\kappa]_K$ as we have explained in the last proposition.

\paragraph{Calculation of $[\widetilde{B}^{\mathcal{A}_N}]_{G/N}(\kappa_1,\kappa_2)$.}
For $\kappa_i\in \mathfrak{k},$ $i\in \{1,2\},$ we can
define on $G$ a vector field associated to $\kappa_i$ through the
formula $X_{\kappa_i}(g)=g. \kappa_i^{\mathcal{A}_N}(e_G).$ This field
is a $G$-invariant horizontal field on the principal bundle $G\rightarrow K$ with principal connection $\mathcal{A}_N$. It follows that $[X_{\kappa_1},X_{\kappa_2}]$ is $G$-invariant and we have  $\left[X_{\kappa_1},X_{\kappa_2}\right](g)
=g\left(\left[X_{\kappa_1},X_{\kappa_2}\right](e_G)\right)=g[\kappa_1^{\mathcal{A}_N}(e_G),\kappa_2^{\mathcal{A}_N}(e_G)].$

\vspace{.3cm}
Using the general formula (\ref{curvatura}) in Appendix \ref{appendixb} we have that

$$[\widetilde{B}^{\mathcal{A}_N}]_K(\kappa_1,\kappa_2)=\left[[e_G,B^{\mathcal{A}_N}(e_G)
(\kappa_1^{\mathcal{A}_N}(e_G),\kappa_2^{\mathcal{A}_N}(e_G))]_N\right]_K=
\left[[e_G,-\mathcal{A}_N\left([X_{\kappa_1},X_{\kappa_2}](e_G)\right)]_N\right]_K=$$

$$\left[\left[e_G,-\mathcal{A}_N\left(e_G[\kappa_1^{\mathcal{A}_N}(e_G),
\kappa_2^{\mathcal{A}_N}(e_G)]\right)\right]_N\right]_K
=
\left[\left[e_G,-\mathcal{A}_N\left([\kappa_1^{\mathcal{A}_N}(e_G),
\kappa_2^{\mathcal{A}_N}(e_G)]\right)\right]_N\right]_K
,$$

\vspace{.3cm} where we have used the
definition of the bracket in $\mathfrak{g}$ and taken into account that the $X_{\kappa_i}$ are left invariant vector fields on $G$
such that $X_{\kappa_i}(e_G)=\kappa_i^{\mathcal{A}_N}(e_G),$ for $i=1,2.$
Finally, using the identification
(\ref{qqqq})
with
$\eta \equiv -\mathcal{A}_N\left([\kappa_1^{\mathcal{A}_N}(e_G),
\kappa_2^{\mathcal{A}_N}(e_G)]\right)$
we obtain the identification
\begin{equation}
[\widetilde{B}^{\mathcal{A}_N}]_K(\kappa_1,\kappa_2) \equiv  -\mathcal{A}_N\left([\kappa_1^{\mathcal{A}_N}(e_G),
\kappa_2^{\mathcal{A}_N}(e_G)]\right).
\end{equation}

In order to simplify the notation, we define the bilinear form
$a_N:\mathfrak{k}\times\mathfrak{k}\rightarrow
\widetilde{\mathfrak{n}}/K\equiv\mathfrak{n}$ given by the following
formula,

\begin{equation}
a_N(\kappa,\overline{\kappa})
:= -\mathcal{A}_N
\left([\kappa^{\mathcal{A}_N}(e_G),\overline{\kappa}^{\mathcal{A}_N}(e_G)]\right)
\end{equation}

and we can write
\begin{equation}
[\widetilde{B}^{\mathcal{A}_N}]_K(\kappa,\overline{\kappa}) \equiv a_N(\kappa,\overline{\kappa}).
\end{equation}

\paragraph{Calculation of $[\nabla^{({\mathcal{A}_N},V)}]_{G/N,
\kappa}\eta$.}
Using the definition of the quotient vertical
connection we have
\begin{equation*}
[\nabla^{({\mathcal{A}_N},V)}]_{G/N,
\kappa}\eta\equiv[\nabla^{({\mathcal{A}_N},V)}]_{K,[e_K,
\kappa]_{K}}[\sigma_{\eta}]_K=
\left[\nabla_{X_0}\sigma_{\eta}\right]_K,
\end{equation*}
where
 $X_0$
is the uniquely determined invariant vertical vector field
on the principal bundle
$K \rightarrow \{e_K\}$ with structure group $K,$
 such that
$\mathcal{A}_K(X_0(e_K)) = \kappa ,$
where
$\mathcal{A}_K$
is the uniquely determined principal connection on
the principal bundle
$K \rightarrow \{e_K\},$ see Appendix \ref{appendixa}.\\

\textbf{Notation.}
For the rest of this paper we will often consider some horizontal lift with respect to some connection, at the neutral element $e = e_G$ of $G$, like, for instance,
$\kappa^{\mathcal{A}_N}(e_G).$ In those cases we will often avoid writting
$e_G,$ for simplicity.\\

We are going to use the following lemma.

\vspace{.5cm}
\begin{lem}\label{DerivCovAdjunto}
There is a natural identification
\begin{equation}
 [\nabla^{(\mathcal{A}_N,V)}]_{G/N,\ \kappa}\
\eta
\equiv
[\kappa^{\mathcal{A}_N},\eta].
\end{equation}
\end{lem}
\begin{proof}

 Let $[q(s),\xi_(s)]_G$ be any curve in $\tilde{\mathfrak{g}}.$ Then we have (see definition \ref{dfn214} in Appendix
\ref{appendixa}),
 $$\frac{D[q(s),\xi(s)]_G}{Ds}=\left[q(s),-[\mathcal{A}(q(s),\dot{q}(s)),\xi(s)]+\dot{\xi}(s)\right]_G.$$

Also we have
$$\nabla_X^{(\mathcal{A}_N,V)}v(q_0)=\left.\frac{D}{Dt} g_q(t)v(t_0)\right |_{t=t_0}$$
with $v\in\Gamma^G(V),\ q(t)\in Q$ such that $\dot{q}(t_0)=X(q_0),$
$\tau:V\rightarrow Q,$ $v(t)\in V$ and
 $q(t)=\tau(v(t)),$ with $q_h(t)$ horizontal such that $q(t)=g_q(t).q_h(t),\ g_q(t_0)=e.$

\vspace{.2cm}
 So, considering the case $Q=G/N=K,$ $V\equiv \widetilde{\mathfrak{n}},$
$\tau=\widetilde{\pi}_N:\widetilde{\mathfrak{n}}\rightarrow K,$ with
$\sigma_{\eta}([g]_N)=$ $=[g]_N.[e_G,\eta]_N,$ and $k(t)=[g(t)]_N,$
$\dot{k}(t_0)=[\dot{g}(t_0)]_N=\kappa,$
$v(t)=[g(t),\text{Ad}_{g(t)}\eta]_N,$ where $g(t_0)=e,$ we can write

$\displaystyle[\nabla^{({\mathcal{A}_N},V)}]_{K,\ \kappa}\
\eta=\left[\nabla_{Y_0}\sigma_{\eta}\right]_K=\left[\left.\frac{D}{Dt}k(t).
\sigma_{\eta}(e_K)\right |_{t=t_0}\right]_K.$ So, we have

\begin{eqnarray*}
[\nabla^{({\mathcal{A}_N},V)}]_{G/N,\ \kappa}\ \eta &=
&\left[\left.\frac{D}{Dt}k(t).
\sigma_{\eta}(e_K)\right |_{t=t_0}\right]_K=\left[\left.\frac{D}{Dt}[g(t)]_N.[e,\eta]_N\right |_{t=t_0}\right]_K=\\
 &=& \left[\left.\frac{D}{Dt}[g(t),\text{Ad}_{g(t)}\eta]_N\right |_{t=t_0}\right]_K
\end{eqnarray*}

where the last identity follows from the definition of the action of
$K=G/N$ over $\widetilde{\mathfrak{n}}.$

\vspace{.2cm} Now, applying the definition of covariant derivative
in the adjoint bundle, we have

\vspace{.2cm}
$\displaystyle\left[\left.\frac{D}{Dt}[g(t),\text{Ad}_{g(t)}\eta]_N\right |_{t=t_0}\right]_K=\left[\left.\left[g(t),-\left[\mathcal{A}_N(g(t),\dot{g}(t)),\text{Ad}_{g(t)}\eta
\right]+(\text{Ad}_{g(t)}\eta)\dot{}\
\right |_{t=t_0}\right]_N\right]_K.$

\vspace{.3cm} If we choose $g(t)$ horizontal with respect to
$\mathcal{A}_N$ such that $g(t_0)=e,$ we obtain the formula
\begin{equation}
 [\nabla^{(\mathcal{A}_N,V)}]_{G/N,\ \kappa}\
\eta=\left[[e,[\kappa^{\mathcal{A}_N},\eta]]_N\right]_K
\end{equation}
where
$\kappa^{\mathcal{A}_N}$ is the horizontal lift of $\kappa$ in $e_G$
to $G.$\\

Using (\ref{qqqq})
we obtain the identification
$\displaystyle
 [\nabla^{(\mathcal{A}_N,V)}]_{G/N,\ \kappa}\
\eta
\equiv
[\kappa^{\mathcal{A}_N},\eta].$
\end{proof}

\vspace{.5cm}
We define the bilinear form
$b_N:\mathfrak{k}\times\mathfrak{n}\rightarrow
\widetilde{\mathfrak{n}}/K\equiv\mathfrak{n}$ given by
\begin{equation}
b_N(\kappa,\eta):= [\kappa^{\mathcal{A}_N},\eta].
\end{equation}
Then we have an explicit formula for the Lie bracket
\begin{align}\label{corchg2etapas}
\left[\kappa\oplus\eta,\overline{\kappa}\oplus\overline{\eta}\right]
=&
\left[\kappa,\overline{\kappa}\right]\oplus
[\nabla^{(\mathcal{A}_N,V)}]_{K,\kappa}\ \overline{\eta}-
[\nabla^{(\mathcal{A}_N,V)}]_{K,\overline{\kappa}}\
\eta-[\widetilde{B}^{\mathcal{A}_N}]_K(\kappa,\overline{\kappa})+
\left[\eta,\overline{\eta}\right] \nonumber \\
=&
\left[\kappa,\overline{\kappa}\right]\oplus\left[[e,[\kappa^{\mathcal{A}_N},
\overline{\eta}]]_N\right]_K
-
\left[[e,[\overline{\kappa}^{\mathcal{A}_N},\eta]]_N\right]_K \nonumber \\
&-\left[\left[e_G,-\mathcal{A}_N\left([\kappa_1^{\mathcal{A}_N}(e_G),
\kappa_2^{\mathcal{A}_N}(e_G)]\right)\right]_N\right]_{G/N}+
[\eta,\overline{\eta}] \nonumber \\
\end{align}
or equivalently, using the defined bilineal forms,

\begin{align}\label{corchg2etapasbilineales}
 \left[\kappa\oplus\eta,\overline{\kappa}\oplus\overline{\eta}\right]
\equiv
[\kappa,\overline{\kappa}]\oplus
b_N(\kappa,\overline{\eta})
-b_N(\overline{\kappa},\eta)-a_N(\kappa,\overline{\kappa})+[\eta,\overline{\eta}].
\end{align}

\vspace{.5cm} \subsection{The Lie bracket by several stages.}

Now we will obtain a formula that generalizes
(\ref{corchg2etapasbilineales}) for the case of several stages, in fact, we will show that it can be done by a repeated application of (\ref{corchg2etapasbilineales}).
 This will be one of the main results of this paper.
Another important result will be the application of this generalization for reducing a nonholonomic system by several stages.

\vspace{.5cm}
\begin{thm} \label{formcorchete}
Let $G$ be a Lie group and let a chain of $n$
normal subgroups $N_j\lhd
N_{j-1}$ for $1\leq j\leq n+1,$ where $N_0=G$
and $N_{n+1}=\{e\}.$
Denote $N_{(j-1,j)}$ the groups $N_{j-1}/N_j,$
 $\mathfrak{n}_j$ the Lie algebra of the subgroup $N_j$
and $\mathfrak{n}_{(j-1, j)}$
 the Lie algebra of $N_{(j-1,j)}.$
Note that since $\mathfrak{n}_{n+1} = \{0\}$ we have
$\mathfrak{n}_{(n, n+1)} = \mathfrak{n}_n.$

Consider a linear identification as vector spaces
\begin{eqnarray}
\label{descompg3} \mathfrak{g} =
\mathfrak{n}_{(0,1)}\oplus\mathfrak{n}_{(1,2)}\oplus...\oplus\mathfrak{n}_{(n-1,n)}\oplus\mathfrak{n}_{(n,n+1)}
\end{eqnarray}
such that
$\mathfrak{n}_{j-1} = \mathfrak{n}_{(j-1,j)}\oplus \mathfrak{n}_j$
for all $1 \leq j \leq n+1.$

Let
$\displaystyle\bigoplus_{i=0}^{n}\eta^{(i,i+1)}$
and
 $\displaystyle\bigoplus_{j=0}^{n}\overline{\eta}^{(j,j+1)}$ be two elements of
 $\mathfrak{g},$ where $\eta^{(i,i+1)}\in \mathfrak{n}_{(i,i+1)},$
  $\overline{\eta}^{(j,j+1)}\in \mathfrak{n}_{(j,j+1)}$, for $i,j = 0,...,n.$

For each $j = 1,...,n+1$, consider the left principal bundle $N_{j-1}\rightarrow N_{(j-1,j)}=N_{j-1}/N_j$ with structure group $N_j$ and with
connection $\mathcal{A}_{N_j}$, defined in a similar way as it was done for the principal bundle
$G \rightarrow G/N$ with structure group $N$ with connection $A_N$, at the beginning of section \ref{sect:corcheteetapas}, using the same $G$-invariant metric, which will be, naturally,
$N_{j-1}$-invariant.

\vspace{.3cm} \label{bilineales} Consider the map \hspace{.2cm} $\displaystyle
b_{(N_{j-1},N_j)}:\mathfrak{n}_{(j-1,j)}\times\mathfrak{n}_j\rightarrow\mathfrak{n}_j,$
$j = 1,..,n+1$,
defined by \vspace{.1cm}\newline \vspace{.1cm}$\displaystyle
b_{(N_{j-1},N_j)}(\eta,\overline{\eta}):=[\eta^{\mathcal{A}_{N_j}},\overline{\eta}]\equiv
\left[\left[e,[\eta^{\mathcal{A}_{N_j}},\overline{\eta}]\right]_{N_j}
\right]_{N_{(j-1,j)}},$ where $\eta^{\mathcal{A}_{N_j}}$ is the horizontal
lift of $\eta$ in the bundle $N_{j-1}\rightarrow N_{(j-1,j)}$ at the neutral element.
Note that $\eta^{\mathcal{A}_{N_j}} \in \mathfrak{n}_{j-1}$ and
$\mathfrak{n}_j \subseteq   \mathfrak{n}_{j-1}$ while
 the Lie bracket
$[\eta^{\mathcal{A}_{N_j}},\overline{\eta}]$ is the Lie bracket in $\mathfrak{n}_{j-1}$.

Similarly, we consider
$a_{(N_{j-1},N_j)}:\mathfrak{n}_{(j-1,j)}\times\mathfrak{n}_{(j-1,j)}\rightarrow\mathfrak{n}_j$,
$j = 1,...,n+1$, defined by
\vspace{.1cm} \newline \vspace{.1cm}

$\displaystyle
a_{(N_{j-1},N_j)}(\eta,\overline{\eta}):=-\mathcal{A}_{N_j}(e)\left([\eta^{\mathcal{A}_{N_{j}}},
\overline{\eta}^{\mathcal{A}_{N_{j}}}]\right)\equiv\left[\left[e,-\mathcal{A}_{N_{j}}(e)\left([
\eta^{\mathcal{A}_{N_{j}}},
\overline{\eta}^{\mathcal{A}_{N_{j}}}]\right)\right]_{N_{j}}\right]_{N_{(j-1,j)}}.$

\vspace{.3cm} The map $b_{(N_{j-1},N_j)}$ is identified with a
quotient vertical connection since

\vspace{.2cm}
$\left[\left[e,[\eta^{\mathcal{A}_{N_j}},\overline{\eta}]\right]_{N_j}
\right]_{N_{(j-1,j)}}\equiv[\nabla^{(\mathcal{A}_{N_j},V)}]_{N_{(j-1,j)},\eta}\
\overline{\eta}.$

\vspace{.1cm}
On the other hand, the map $a_{(N_{j-1},N_j)}$ is related
directly with the curvature of the connection $\mathcal{A}_{N_{j}}$,
as follows \vspace{.1cm}\newline $\displaystyle
-\mathcal{A}_{N_{j}}(e)\left([\eta^{\mathcal{A}_{N_j}},
\overline{\eta}^{\mathcal{A}_{N_{j}}}]\right)=B^{\mathcal{A}_{N_{j}}}(e)(\eta^{\mathcal{A}_{N_{j}}},
\overline{\eta}^{\mathcal{A}_{N_{j}}}).$

\vspace{.5cm}

Then we have the following formula for the Lie bracket on
$\mathfrak{g}:$

\begin{align} \label{formcorchetenetapas}
\left[\bigoplus_{i=0}^{n}\eta^{(i,i+1)},\bigoplus_{j=0}^{n}
\overline{\eta}^{(j,j+1)}\right]=\bigoplus_{i=0}^{n}\left([\eta^{(i,i+1)},
\overline{\eta}^{(i,i+1)}]+ \sum_{j=0}^{i-1}\left(-a_{(N_j,N_{j+1})}^{(i,i+1)}(\eta^{(j,j+1)},\overline{\eta}^{(j,j+1)})+\right.\right.
\hspace{11cm} \nonumber \\
\left.\left.
b_{(N_j,N_{j+1})}^{(i,i+1)}\left(\eta^{(j,j+1)},\sum_{k=j+1}^{n}
\overline{\eta}^{(k,k+1)}\right)-b_{(N_j,N_{j+1})}^{(i,i+1)}\left(\overline{\eta}^{(j,j+1)},\sum_{l=j+1}^{n}
\eta^{(l,l+1)}\right) \right)\right), \hspace{12cm}
\end{align}

\vspace{.2cm}

where by definition, for $i = 0,...,n$ and $j = 0,...,n$,
 $b_{(N_{j},N_{j+1})}^{(i,i+1)}$ and $a_{(N_{j},N_{j+1})}^{(i,i+1)}$ are, respectively, the components
of $b_{(N_{j},N_{j+1})}$ and $a_{(N_{j},N_{j+1})}$
on $\mathfrak{n}_{(i,i+1)}$.
\end{thm}

\begin{proof}

\vspace{.5cm} We shall prove this formula by induction.

\vspace{.3cm} If $n=1$ then formula (\ref{formcorchetenetapas}) becomes
$$[\eta^{(0,1)}\oplus\eta^{(1,2)},\overline{\eta}^{(0,1)}\oplus\overline{\eta}^{(1,2)}]=[\eta^{(0,1)},
\overline{\eta}^{(0,1)}]\oplus[
\eta^{(1,2)},\overline{\eta}^{(1,2)}]-a_{(N_0,N_1)}^{(1,2)}(\eta^{(0,1)},\overline{\eta}^{(0,1)})+$$
$$\hspace{4cm} b_{(N_0,N_1)}^{(1,2)}(\eta^{(0,1)},\overline{\eta}^{(1,2)})-
b_{(N_0,N_1)}^{(1,2)}(\overline{\eta}^{(0,1)},
\eta^{(1,2)}) ,$$
and this expression coincides with equation (\ref{corchg2etapasbilineales})
for the case $G=N_0,$  $N_1\equiv N,$
$b_{(N_0,N_1)}\equiv b_N$ and $N_2=\{e\},$ where
$b^{(1,2)}_{(N_0,N_1)}= b_{(N_0,N_1)}\equiv b_N$
 and $a^{(1,2)}_{(N_0,N_1)}=
a_{(N_0,N_1)}\equiv a_N.$

\vspace{.5cm} Now, we shall suppose that formula (\ref{formcorchetenetapas}) is valid in the
case in which we have $n-1$ normal subgroups and we shall see that
the formula holds when the number of subgroups is $n$.

\vspace{.2cm}
Using the formula (\ref{corchg2etapasbilineales}) that we have obtained in the case of one
normal subgroup, with $N_1 \lhd G$ whose Lie algebra can be written as
$\displaystyle\mathfrak{n}_1 = \bigoplus_{i=0}^{1}\mathfrak{n}_{(i,i+1)},$
we have
$$\left[\bigoplus_{i=0}^{n}\eta^{(i,i+1)},\bigoplus_{j=0}^{n}\overline{\eta}^{(j,j+1)}\right]=[\eta^{(0,1)},
\overline{\eta}^{(0,1)}]
\oplus\left(
\left[\bigoplus_{i=1}^{n}\eta^{(i,i+1)},\bigoplus_{j=1}^{n}\overline{\eta}^{(j,j+1)}\right]+\right.$$

\begin{eqnarray}\label{corcheten}
\left. b_{(N_0,N_1)}\left(\eta^{(0,1)},\bigoplus_{j=1}^{n}\overline{\eta}^{(j,j+1)}\right)-b_{(N_0,N_1)}
\left(\overline{
\eta}^{(0,1)},\bigoplus_{i=1}^{n}\eta^{(i,i+1)}\right)-a_{(N_0,N_1)}\left(\eta^{(0,1)},\overline{\eta}^{(0,1)}\right)
\right).
\end{eqnarray}

Using the inductive hypothesis we obtain
\newline
\begin{equation*}\label{corchetenmenosuno}
\left[\bigoplus_{i=1}^{n}\eta^{(i,i+1)},\bigoplus_{j=1}^{n}\overline{\eta}^{(j,j+1)}\right]= \bigoplus_{i=1}^n
\left(\left[\eta^{(i,i+1)},\overline{\eta}^{(i,i+1)}\right]+\right.
\end{equation*}
\begin{equation*}
\left.
\sum_{j=1}^{i-1}
\left(-a_{(N_j,N_{j+1})}^{(i,i+1)}(\eta^{(j,j+1)},\overline{\eta}^{(j,j+1)})+b_{(N_j,N_j+1)}^{(i,i+1)}\left( \eta^{(j,j+1)},
\sum_{k=j+1}^{n}\overline{\eta}^{(k,k+1)}\right)\right.\right.
\end{equation*}
\begin{equation*}
\left.\left.
-b_{(N_j,N_{j+1})}^{(i,i+1)}\left(\overline{\eta}^{(j,j+1)},\sum_{l=j+1}^{n}
\eta^{(l,l+1)}\right)\right)\right).
\end{equation*}

\vspace{.3cm} Replacing this expression in (\ref{corcheten}) and considering that $\displaystyle b_{(N_0,N_1)}(\eta,\overline{\eta})=\bigoplus_{p=1}^n b_{(N_0,N_1)}^{(p,p+1)}(\eta,\overline{\eta})$ and $\displaystyle a_{(N_0,N_1)}(\eta,\overline{\eta})=\bigoplus_{p=1}^n a_{(N_0,N_1)}^{(p,p+1)}(\eta,\overline{\eta}),$ we obtain

\begin{equation*}
\left[\bigoplus_{i=0}^{n}\eta^{(i,i+1)},\bigoplus_{j=0}^{n}\overline{\eta}^{(j,j+1)}\right]=[\eta^{(0,1)},
\overline{\eta}^{(0,1)}]
\oplus\left(
\bigoplus_{i=1}^n
\left(\left[\eta^{(i,i+1)},\overline{\eta}^{(i,i+1)}\right]+\right.\right.
\end{equation*}
\begin{equation*}
\displaystyle\left.\left.\sum_{j=1}^{i-1}
\left(-a_{(N_j,N_{j+1})}^{(i,i+1)}(\eta^{(j,j+1)},\overline{\eta}^{(j,j+1)})+
\right.\right.\right.
\end{equation*}
\begin{equation*}
\left.\left.\left.
b_{(N_j,N_j+1)}^{(i,i+1)}
\left(\eta^{(j,j+1)},
\sum_{k=j+1}^{n}\overline{\eta}^{(k,k+1)}\right)-b_{(N_j,N_{j+1})}^{(i,i+1)}\left(\overline{\eta}^{(j,j+1)},\sum_{l=j+1}^{n}
\eta^{(l,l+1)}\right)\right)\right)\right)+
\end{equation*}
\begin{equation*}
\bigoplus_{p=1}^n \left(b_{(N_0,N_1)}^{(p,p+1)}\left(\eta^{(0,1)},\sum_{s=1}^{n}\overline{\eta}^{(s,s+1)}\right)
- b_{(N_0,N_1)}^{(p,p+1)}
\left(\overline{
\eta}^{(0,1)},\sum_{m=1}^{n}\eta^{(m,m+1)}\right)-
a_{(N_0,N_1)}^{(p,p+1)}\left(\eta^{(0,1)},\overline{\eta}^{(0,1)}\right)
\right).
\end{equation*}

\vspace{.3cm}
Replacing the index  $p$ by  $i$ we can group terms together and we obtain

\begin{equation*}
\left[\bigoplus_{i=0}^{n}\eta^{(i,i+1)},\bigoplus_{j=0}^{n}\overline{\eta}^{(j,j+1)}\right]=[\eta^{(0,1)},
\overline{\eta}^{(0,1)}]
\oplus\left(
\bigoplus_{i=1}^n
\left(\left[\eta^{(i,i+1)},\overline{\eta}^{(i,i+1)}\right]+\right.\right.
\end{equation*}
\begin{equation*}
\displaystyle\left.\left.
\sum_{j=1}^{i-1}
\left(-a_{(N_j,N_{j+1})}^{(i,i+1)}(\eta^{(j,j+1)},\overline{\eta}^{(j,j+1)})+
\right.\right.\right.
\end{equation*}
\begin{equation*}
\left.\left.\left.
b_{(N_j,N_j+1)}^{(i,i+1)}
\left(\eta^{(j,j+1)},\sum_{k=j+1}^{n}\overline{\eta}^{(k,k+1)}\right)-b_{(N_j,N_{j+1})}^{(i,i+1)}
\left(\overline{\eta}^{(j,j+1)},\sum_{l=j+1}^{n}\eta^{(l,l+1)}\right)\right)
\right.\right.
\end{equation*}
\begin{equation*}
\left.\left.
-
a_{(N_0,N_1)}^{(i,i+1)}\left(\eta^{(0,1)},\overline{\eta}^{(0,1)}\right)+b_{(N_0,N_1)}^{(i,i+1)}\left(\eta^{(0,1)},\sum_{s=1}^{n}\overline{\eta}^{(s,s+1)}\right)
- b_{(N_0,N_1)}^{(i,i+1)}
\left(\overline{\eta}^{(0,1)},\sum_{m=1}^{n}\eta^{(m,m+1)}\right)\right)\right)=
\end{equation*}


\begin{equation*}
[\eta^{(0,1)},
\overline{\eta}^{(0,1)}]
\oplus\left(
\bigoplus_{i=1}^n
\left(\left[\eta^{(i,i+1)},\overline{\eta}^{(i,i+1)}\right]+
\sum_{j=0}^{i-1}
\left(-a_{(N_j,N_{j+1})}^{(i,i+1)}(\eta^{(j,j+1)},\overline{\eta}^{(j,j+1)})+
\right.\right.\right.
\end{equation*}
\begin{equation*}
\left.\left.\left.
b_{(N_j,N_j+1)}^{(i,i+1)}
\left(\eta^{(j,j+1)},\sum_{k=j+1}^{n}\overline{\eta}^{(k,k+1)}\right)-b_{(N_j,N_{j+1})}^{(i,i+1)}
\left(\overline{\eta}^{(j,j+1)},\sum_{l=j+1}^{n}\eta^{(l,l+1)}\right)\right)
\right)\right)=
\end{equation*}

\begin{equation*}
\bigoplus_{i=0}^n
\left(\left[\eta^{(i,i+1)},\overline{\eta}^{(i,i+1)}\right]+
\sum_{j=0}^{i-1}
\left(-a_{(N_j,N_{j+1})}^{(i,i+1)}(\eta^{(j,j+1)},\overline{\eta}^{(j,j+1)})+
\right.\right.
\end{equation*}
\begin{equation*}
\left.\left.
b_{(N_j,N_j+1)}^{(i,i+1)}
\left(\eta^{(j,j+1)},\sum_{k=j+1}^{n}\overline{\eta}^{(k,k+1)}\right)-b_{(N_j,N_{j+1})}^{(i,i+1)}
\left(\overline{\eta}^{(j,j+1)},\sum_{l=j+1}^{n}\eta^{(l,l+1)}\right)
\right)\right).
\end{equation*}

\vspace{.5cm} In other words,

\begin{eqnarray*}
\left[\bigoplus_{i=0}^{n}\eta^{(i,i+1)},\bigoplus_{j=0}^{n}
\overline{\eta}^{(j,j+1)}\right]=\bigoplus_{i=0}^{n}\left([\eta^{(i,i+1)},
\overline{\eta}^{(i,i+1)}]+ \sum_{j=0}^{i-1}\left(-a_{(N_j,N_{j+1})}^{(i,i+1)}(\eta^{(j,j+1)},\overline{\eta}^{(j,j+1)})+
\right.\right.\hspace{11cm}\nonumber \\
\left.\left. b_{(N_j,N_{j+1})}^{(i,i+1)}\left(\eta^{(j,j+1)},\sum_{k=j+1}^{n}
\overline{\eta}^{(k,k+1)}\right)-b_{(N_j,N_{j+1})}^{(i,i+1)}\left(\overline{\eta}^{(j,j+1)},\sum_{l=j+1}^{n}
\eta^{(l,l+1)}\right) \right)\right). \hspace{10cm}
\end{eqnarray*}

\end{proof}

\vspace{1cm}
In particular, for  $n=2$ and  $n=3,$
we obtain the following expressions

$$
\left[\eta^{(0,1)}\oplus\eta^{(1,2)}\oplus\eta^{(2,3)},\overline{\eta}^{(0,1)}\oplus\overline{\eta}^{(1,2)}\oplus
\overline{\eta}^{(2,3)}\right]
=[\eta^{(0,1)},\overline{\eta}^{(0,1)}]\oplus[\eta^{(1,2)},\overline{\eta}^{(1,2)}]+b_{(N_0,N_1)}^{(1,2)}(\eta^{(0,1)},
\overline{\eta}^{(1,2)}) +$$
$$ b_{(N_0,N_1)}^{(1,2)}(\eta^{(0,1)},\overline{\eta}^{(2,3)})-
b_{(N_0,N_1)}^{(1,2)}(\overline{\eta}^{(0,1)},\eta^{(1,2)})-
b_{(N_0,N_1)}^{(1,2)}(\overline{\eta}^{(0,1)},\eta^{(2,3)})-
a_{(N_0,N_1)}^{(1,2)}(\eta^{(0,1)},\overline{\eta}^{(0,1)})$$ $$\oplus[\eta^{(2,3)},\overline{\eta}^{(2,3)}]+
b_{(N_0,N_1)}^{(2,3)}(\eta^{(0,1)},\overline{\eta}^{(1,2)})+b_{(N_0,N_1)}^{(2,3)}(\eta^{(0,1)},\overline{\eta}^{(2,3)})
-b_{(N_0,N_1)}^{(2,3)}
(\overline{\eta}^{(0,1)},\eta^{(1,2)})$$
$$-b_{(N_0,N_1)}^{(2,3)}(\overline{\eta}^{(0,1)},\eta^{(2,3)})
-a_{(N_0,N_1)}^{(2,3)}
(\eta^{(0,1)},\overline{\eta}^{(0,1)})+
b_{(N_1,N_2)}^{(2,3)}(\eta^{(1,2)},\overline{\eta}^{(2,3)}) $$
$$-b_{(N_1,N_2)}^{(2,3)}(\overline{\eta}^{(1,2)},\eta^{(2,3)})-
a_{(N_1,N_2)}^{(2,3)}(\eta^{(1,2)},\overline{\eta}^{(1,2)}),
$$\\

and

$$\left[\eta^{(0,1)}\oplus\eta^{(1,2)}\oplus\eta^{(2,3)}\oplus\eta^{(3,4)},\overline{\eta}^{(0,1)}\oplus\overline{\eta}^{(1,2)}
\oplus\overline{\eta}^{(2,3)}\oplus\overline{\eta}^{(3,4)}\right]=[\eta^{(0,1)},\overline{\eta}^{(0,1)}]\oplus[\eta^{(1,2)},
\overline{\eta}^{(1,2)}]+$$
$$b^{(1,2)}_{(N_0,N_1)}(\eta^{(0,1)},\overline{\eta}^{(1,2)})+b^{(1,2)}_{(N_0,N_1)}(\eta^{(0,1)},
\overline{\eta}^{(2,3)})+b^{(1,2)}_{(N_0,N_1)}(\eta^{(0,1)},\overline{\eta}^{(3,4)})-b^{(1,2)}_{(N_0,N_1)}(\overline{\eta}^{(0,1)},
\overline{\eta}^{(1,2)})$$ $$-b^{(1,2)}_{(N_0,N_1)}(\overline{\eta}^{(0,1)},\overline{\eta}^{(2,3)})-b^{(1,2)}_{(N_0,N_1)}(
\overline{\eta}^{(0,1)},
\overline{\eta}^{(3,4)})-a^{(1,2)}_{(N_0,N_1)}(\eta^{(0,1)},\overline{\eta}^{(0,1)})\oplus
b^{(2,3)}_{(N_0,N_1)}(\eta^{(0,1)},\overline{\eta}^{(1,2)})+$$
$$b^{(2,3)}_{(N_0,N_1)}(\eta^{(0,1)},\overline{\eta}^{(2,3)})+
b^{(2,3)}_{(N_0,N_1)}
(\eta^{(0,1)},\overline{\eta}^{(3,4)})-b^{(2,3)}_{(N_0,N_1)}(\overline{\eta}^{(0,1)},\eta^{(1,2)})-b^{(2,3)}_{(N_0,N_1)}(\overline{\eta}^{(0,1)},\eta^{(2,3)})
$$
$$-b^{(2,3)}_{(N_0,N_1)}(\overline{\eta}^{(0,1)},\eta^{(3,4)})-a^{(2,3)}_{(N_0,N_1)}(\eta^{(0,1)},\overline{\eta}^{(0,1)})+[\eta^{(2,3)},
\overline{\eta}^{(2,3)}]+ b^{(2,3)}_{(N_1,N_2)}(\eta^{(1,2)},\overline{\eta}^{(2,3)})+$$ $$b^{(2,3)}_{(N_1,N_2)}(\eta^{(1,2)},\overline{\eta}^{(3,4)})
-b^{(2,3)}_{(N_1,N_2)}(\overline{\eta}^{(1,2)},\eta^{(2,3)})-b^{(2,3)}_{(N_1,N_2)}(\overline{\eta}^{(1,2)},\eta^{(3,4)})
-a^{(2,3)}_{(N_1,N_2)}(\eta^{(1,2)},\overline{\eta}^{(1,2)})$$
$$\oplus[\eta^{(3,4)},\overline{\eta}^{(3,4)}]+
b^{(3,4)}_{(N_0,N_1)}(\eta^{(0,1)},\overline{\eta}^{(1,2)})+b^{(3,4)}_{(N_0,N_1)}(\eta^{(0,1)},\overline{\eta}^{(2,3)})+
b^{(3,4)}_{(N_0,N_1)}(\eta^{(0,1)},
\overline{\eta}^{(3,4)})$$
$$-b^{(3,4)}_{(N_0,N_1)}(\overline{\eta}^{(0,1)},\eta^{(1,2)})-b^{(3,4)}_{(N_0,N_1)}(\overline{\eta}^{(0,1)}, \eta^{(2,3)})
-b^{(3,4)}_{(N_0,N_1)}(\overline{\eta}^{(0,1)},\eta^{(3,4)})-a^{(3,4)}_{(N_0,N_1)}(\eta^{(0,1)},\overline{\eta}^{(0,1)})+$$
$$ b^{(3,4)}_{(N_1,N_2)}(\eta^{(1,2)},\overline{\eta}^{(2,3)})+b^{(3,4)}_{(N_1,N_2)}(\eta^{(1,2)},\overline{\eta}^{(3,4)})
-b^{(3,4)}_{(N_1,N_2)}(\overline{\eta}^{(1,2)},\eta^{(2,3)})-b^{(3,4)}_{(N_1,N_2)}(\overline{\eta}^{(1,2)},\eta^{(3,4)})
$$
$$-a^{(3,4)}_{(N_1,N_2)}(\eta^{(1,2)},\overline{\eta}^{(1,2})+b^{(3,4)}_{(N_2,N_3)}(\eta^{(2,3)},\overline{\eta}^{(3,4)})
-b^{(3,4)}_{(N_2,N_3)}(\overline{\eta}^{(2,3)},\eta^{(3,4)})-a^{(3,4)}_{(N_2,N_3)}(\eta^{(2,3)},\overline{\eta}^{(2,3)}).
$$

\section{The Euler-d'Alembert-Poincar{\'e} equations by several stages} \label{sect:Euler-d'Alem-Poin}

\vspace{.5cm} \hspace{.5cm} For Euler-Poincar\'e equations and Euler-d'Alembert-Poincar\'e equations, see \cite{marsden3}  and  \cite{CMR01b}. These equations are the particular case of
Lagrange-Poincar{\'e} equations and Lagrange-d'Alembert-Poincar{\'e} equations, respectively, for the case in which $Q=G$ is the trivial bundle whose base is a single point. In this section we are going to study the Euler-d'Alembert-Poincar{\'e} equations by \textit{several}  stages and in section \ref{sect:Lag-d'Alem-Poin} we will study the case of a trivial bundle $Q = X \times G.$ From this, the case of a general bundle $Q$ will follow by glueing together principal bundle charts.

\subsection{The Euler-Poincar{\'e} equations by several stages.}

We are going to use the notation of theorem \ref{formcorchete}.

\vspace{.5cm}
Let $\displaystyle
v(t)=\bigoplus_{i=0}^{n}\eta^{(i,i+1)}(t)\in\mathfrak{g}$
be a curve and consider a variation of this curve \newline $\displaystyle\delta v(t)=\bigoplus_{i=0}^n\delta\eta^{(i,i+1)}(t).$
In order to write Euler-Poincar\'{e} equations
we need to calculate variations $\delta v(t)$ satisfying $\delta v(t)=\dot{\omega}(t)+[v(t),\omega(t)]$, where
$\omega(t)\in\mathfrak{g}$ is a curve such that
$\omega(t_0)=\omega(t_1)=0.$
Let us write
\begin{eqnarray}\label{omega}
\omega(t)=\bigoplus_{j=0}^{n}\xi^{(j,j+1)}(t)
\end{eqnarray}
with $\xi^{(j,j+1)}(t)\in\mathfrak{n}_{(j,j+1)},$ such that
$\xi^{(j,j+1)}(t_0)=\xi^{(j,j+1)}(t_1)=0,\
\forall\ 0\leq j\leq n.$
Then,

\begin{eqnarray*}
\delta v(t) &=& \bigoplus_{i=0}^n\delta\eta^{(i,i+1)}(t)=\dot{\omega}+[v,\omega]=\left(\bigoplus_{j=0}^{n}\dot{\xi}^{(j,j+1)}(t)\right)+
\left[\bigoplus_{i=0}^{n}\eta^{(i,i+1)}(t),
\bigoplus_{j=0}^{n}\xi^{(j,j+1)}(t)\right] \nonumber \\
 &=& \left(\bigoplus_{j=0}^{n}\dot{\xi}^{(j,j+1)}(t)\right)+
\bigoplus_{i=0}^{n}\left([\eta^{(i,i+1)},
\xi^{(i,i+1)}]+ \sum_{j=0}^{i-1}\left(-a_{(N_j,N_{j+1})}^{(i,i+1)}(\eta^{(j,j+1)},\xi^{(j,j+1)})+ \right.\right.\nonumber \\
& & \left.\left.
 b_{(N_j,N_{j+1})}^{(i,i+1)}\left(\eta^{(j,j+1)},\sum_{k=j+1}^{n}
\xi^{(k,k+1)}\right)-b_{(N_j,N_{j+1})}^{(i,i+1)}\left(\xi^{(j,j+1)},\sum_{l=j+1}^{n}
\eta^{(l,l+1)}\right) \right)\right)=\nonumber \\
& & =\bigoplus_{i=0}^{n}\left(\dot{\xi}^{(i,i+1)}(t)+[\eta^{(i,i+1)},
\xi^{(i,i+1)}]+ \sum_{j=0}^{i-1}\left(-a_{(N_j,N_{j+1})}^{(i,i+1)}(\eta^{(j,j+1)},\xi^{(j,j+1)})+ \right.\right.\nonumber \end{eqnarray*}
\begin{eqnarray*}
& & \left.\left.
 b_{(N_j,N_{j+1})}^{(i,i+1)}\left(\eta^{(j,j+1)},\sum_{k=j+1}^{n}
\xi^{(k,k+1)}\right)-b_{(N_j,N_{j+1})}^{(i,i+1)}\left(\xi^{(j,j+1)},\sum_{l=j+1}^{n}
\eta^{(l,l+1)}\right) \right)\right)
\end{eqnarray*}

\vspace{.5cm} From the previous identities we can deduce the following equalities, for $0\leq i\leq n:$

\begin{eqnarray*}
  \delta\eta^{(i,i+1)}(t) & = & \dot{\xi}^{(i,i+1)}(t)+[\eta^{(i,i+1)},
\xi^{(i,i+1)}]+ \sum_{j=0}^{i-1}\left(-a_{(N_j,N_{j+1})}^{(i,i+1)}(\eta^{(j,j+1)},\xi^{(j,j+1)})+ \right.
\end{eqnarray*}
\begin{eqnarray}\label{variacEPN}
& & \left.
 b_{(N_j,N_{j+1})}^{(i,i+1)}\left(\eta^{(j,j+1)},\sum_{k=j+1}^n
\xi^{(k,k+1)}\right)-b_{(N_j,N_{j+1})}^{(i,i+1)}\left(\xi^{(j,j+1)},\sum_{l=j+1}^{n}
\eta^{(l,l+1)}\right)\right).
\end{eqnarray}

\vspace{.5cm} Since we are assuming that
$\displaystyle\mathfrak{g}\equiv\bigoplus_{i=0}^{n}\mathfrak{n}_{(i,i+1)}$  we have
$\displaystyle\mathfrak{g}^*\equiv\bigoplus_{j=1}^{n}\mathfrak{n}^*_{(j,j+1)}.$

As before, let $v(t)\in\mathfrak{g}$ with $\displaystyle
v(t)=\bigoplus_{i=0}^{n}\eta^{(i,i+1)}(t),$ then
$\displaystyle \frac{\partial l}{\partial v}(t)\in\mathfrak{g}^*.$ Let

\begin{equation}\label{BetaComponents}
\beta(t)=\frac{\partial l}{\partial v}(t)=\bigoplus_{j=0}^{n}\frac{\partial
l}{\partial\eta^{(j,j+1)}}(t)=\bigoplus_{j=0}^{n}\beta_{(j,j+1)}(t),
\end{equation}
where $\beta_{(j,j+1)}(t)\in\mathfrak{n}^*_{(j,j+1)},$ for $0\leq j\leq n.$

\vspace{.5cm} Recall that the Euler-Poincar{\'e} equations
$\dot{\beta}=\text{ad}^*_v\beta$ ( or, equivalently,
$\dot{\beta}(\xi)=\langle\beta,[v,\xi]\rangle$) are obtained from
$\displaystyle\delta\int_{t_0}^{t_1}l(v)dt=0$ with $\delta
v=\dot{\omega}+[v,\omega],$
where $\omega(t)$
represents an arbitrary curve on the Lie algebra $\mathfrak{g}$
satisfying  $\omega(t_0) =  \omega(t_1) = 0.$
Using the decomposition of $v(t),$ the expression (\ref{omega}) for
$\omega$ and (\ref{BetaComponents})
for $\beta$, we obtain the Euler-Poincar\'{e} equations by the condition

$$\left\langle \bigoplus_{j=0}^{n}\dot{\beta}_{(j,j+1)},\bigoplus_{k=0}^{n}\xi^{(k,k+1)}\right\rangle=\left\langle\bigoplus_{m=0}^{n}
\beta_{(m,m+1)},
\left[\bigoplus_{i=0}^{n}\eta^{(i,i+1)},\bigoplus_{k=0}^{n}\xi^{(k,k+1)}\right]\right\rangle$$

\vspace{.2cm} for all $\xi^{(j,j+1)}\in\mathfrak{n}_{(j,j+1)}.$

\vspace{.5cm} Now using the
formula for the Lie bracket obtained in theorem \ref{formcorchete} we have

\begin{eqnarray*}
& & \sum_{j=0}^{n}\dot{\beta}_{(j,j+1)}(\xi^{(j,j+1)})=
\left\langle\bigoplus_{m=0}^{n}\beta_{(m,m+1)},\bigoplus_{i=0}^{n}\left([\eta^{(i,i+1)},
\xi^{(i,i+1)}]+ \right.\right.\nonumber\\
& & \left.\left. \displaystyle\sum_{j=0}^{i-1}
\left(-a_{(N_j,N_{j+1})}^{(i,i+1)}(\eta^{(j,j+1)},\xi^{(j,j+1)})+  b_{(N_j,N_{j+1})}^{(i,i+1)}\left(\eta^{(j,j+1)},\sum_{k=j+1}^{n}
\xi^{(k,k+1)}\right)\right.\right.\right.\nonumber \\
& & \left.\left.\left.
-b_{(N_j,N_{j+1})}^{(i,i+1)}\left(\xi^{(j,j+1)},\sum_{l=j+1}^{n}
\eta^{(l,l+1)}\right) \right)\right)\right\rangle,
\end{eqnarray*}
 for all $\xi^{(j,j+1)}\in\mathfrak{n}_{(j,j+1)}.$

\vspace{.3cm} Then, the following equivalent condition is obtained

\begin{eqnarray}\label{E-Pgral}
& & \sum_{j=0}^{n}\dot{\beta}_{(j,j+1)}(\xi^{(j,j+1)})=
\sum_{i=0}^n\left\langle\beta_{(i,i+1)},[\eta^{(i,i+1)},
\xi^{(i,i+1)}]+ \sum_{j=0}^{i-1}\left(-a_{(N_j,N_{j+1})}^{(i,i+1)}(\eta^{(j,j+1)},\xi^{(j,j+1)})+ \right.\right.\nonumber \\
& & \left.\left.
 b_{(N_j,N_{j+1})}^{(i,i+1)}\left(\eta^{(j,j+1)},\sum_{k=j+1}^{n}
\xi^{(k,k+1)}\right)-b_{(N_j,N_{j+1})}^{(i,i+1)}\left(\xi^{(j,j+1)},\sum_{l=j+1}^{n}
\eta^{(l,l+1)}\right) \right)\right\rangle,
\end{eqnarray}
 for all $\xi^{(j,j+1)}\in\mathfrak{n}_{(j,j+1)}.$

\vspace{.5cm}
 The terms $a_{(N_j,N_{j+1})}^{(i,i+1)}(\eta^{(j,j+1)},\xi^{(j,j+1)})$ vanishes since $\xi^{(j,j+1)}= 0$ because
 $j <i$. Then we have the \textit{\textbf{Euler-Poincar\'{e} equations by stages}}

\begin{equation}\label{E-Pstages1}
\left\{ \begin{array}{rcl}
\displaystyle \dot{\beta}_{(i,i+1)}\mid_{\mathfrak{n}(i,i+1)}&
=&\displaystyle \beta_{(i,i+1)}\left.\left([ \eta^{(i,i+1)},\ .\
]+\sum_{j=0}^{i-1}b^{(i,i+1)}_{(N_j,N_{j+1})}
(\eta^{(j,j+1)},\ .\ )\right)\right | _{\mathfrak{n}_{(i,i+1)}} \\
\beta_{(i,i+1)}&=&\displaystyle\frac{\partial l}{\partial \eta^{(i,i+1)}}(v)\\
\end{array}\right.
\end{equation}

\vspace{.5cm} for $0\leq i\leq n$ and $v\in\mathfrak{g},$ or equivalently,

\vspace{.3cm}
\begin{equation}\label{E-Pstages2}
\left\{ \begin{array}{rcl}
\displaystyle \dot{\beta}_{(i,i+1)}\mid_{\mathfrak{n}_{(i,i+1)}}&
=&\displaystyle (\text{ad}^*_{\eta^{(i,i+1)
}}\beta_{(i,i+1)})\mid_{\mathfrak{n}{(i,i+1)}}+ \\
& & \displaystyle\beta_{(i,i+1)}
\left.\left(\sum_{j=0}^{i-1}b^{(i,i+1)}_{(N_{j},N_{j+1})}(\eta^{(j,j+1)},\ .\ )\right)\right|_{\mathfrak{n}_{(i,i+1)}}  \\
\beta_{(i,i+1)}&=&\displaystyle\frac{\partial l}{\partial \eta^{(i,i+1)}}\\
\end{array}
\right.
\end{equation}

\vspace{.2cm} where $\eta^{(j,j+1)}\in\mathfrak{n}_{(j,j+1)}.$

\vspace{.7cm}
\textbf{Example.}
The simplest example of equations (\ref{E-Pstages2})  occurs when
 $G$  is a direct product $G=N_{(0,1)}\times N_{(1,2)}\times...\times
N_{(n,n+1)}.$
We choose our $G$-invariant metric on $G$ as being a product of invariant metrics on each factor.
Then each
$\mathcal{A}_{N_j}$ is the trivial connection in the product bundle
$N_{j-1} = N_{(j-1, j)} \times N_j \rightarrow N_{(j-1, j)}$.
As we have noted in the definition of $b_{(N_{j-1}, N_j)}$,
$\eta^{\mathcal{A}_{N_j}} \in \mathfrak{n}_{j-1}$ and
$\mathfrak{n}_j \subseteq   \mathfrak{n}_{j-1}$ while
 the Lie bracket
$[\eta^{\mathcal{A}_{N_j}},\overline{\eta}]$ is the Lie bracket in
$\mathfrak{n}_{j-1} = \mathfrak{n}_{(j-1,j)}\times \mathfrak{n}_j$. The triviality of
$\mathcal{A}_{N_j}$ implies that the component of $\eta^{\mathcal{A}_{N_j}}$  corresponding to the factor $\mathfrak{n}_j$ vanishes,
so in the case of the present example $b_{(N_{j-1}, N_j)} = 0$.

Then the Euler-Poincar{\'e} equations become

\vspace{.3cm} $\displaystyle \dot{\beta}_{(i,i+1)}\mid_{\mathfrak{n}(i,i+1)}
=\displaystyle \text{ad}^*_{\eta^{(i,i+1)}}\beta_{(i,i+1)}\mid_{\mathfrak{n}(i,i+1)},$ \hspace{1cm} with $\displaystyle\beta_{(i,i+1)}=\frac{\partial l}{\partial \eta^{(i,i+1)}}.$\\

In the particular case in which
$l (\eta^{(0,1)},...,\eta^{(n,n+1)}) = l_0 (\eta^{(0,1)}) +...+ l_{n} (\eta^{(n,n+1)}),$
where $\eta^{(i,i+1)} \in \mathfrak{n}_{(i,i+1)},$ $i = 0,...,n,$ the previous system of equations (\ref{E-Pstages2}) becomes uncoupled.
If, in addition,  some $N_{(i, i+1)},$ for
$i = 0,...,n,$
is abelian then the corresponding quantity
$\beta_{(i,i+1)}$ is a constant of motion.

\vspace{.3cm} \subsection{The Euler-d'Alembert-Poincar{\'e} equations by
several stages.}\label{EDPBSS}

\vspace{.5cm} In this section we are going to use some results and notation from
\cite{CMR01b}. Also, we are going to use the notation presented in
Appendix \ref{appendixb}, for the particular case
$Q = G.$

\vspace{.1cm} Let
$\displaystyle\mathfrak{g}\equiv\mathfrak{n}_{(0,1)}\oplus\mathfrak{n}_{(1,2)}\oplus...\oplus\mathfrak{n}_{(n-1,n)}\oplus\mathfrak{n}_{(n,n+1)}$
as before, and let $\mathcal{D}$ be a left invariant distribution on
$G,$
which represents the nonholonomic constraint.
In particular,
$\mathcal{D}_e\subset\mathfrak{g}.$
The vertical bundle $\mathcal{V}$
is such that $\mathcal{V}_e=T_eG\equiv \mathfrak{g}.$
We shall also assume the
dimension hypothesis (see equation \ref{hipdimension} in Appendix \ref{appendixb}).

To write the Euler-d'Alembert-Poincar\'e equations by several stages we
will consider an easy particular case first, and the general case
later.

\paragraph{Particular case.} We are going to assume that
$$\mathcal{D}:=\mathcal{D}_e\equiv\mathcal{D}_{(0,1)}\oplus\mathcal{D}_{(1,2)}\oplus
...\oplus\mathcal{D}_{(n-1,n)}\oplus\mathcal{D}_{(n,n+1)}$$ where
$\mathcal{D}_{(i,i+1)}=\mathcal{D}\cap\mathfrak{n}_{(i,i+1)},$ with
$0\leq i\leq n.$

\vspace{.5cm} Since $Q=G$ we have that
$\mathcal{S}=\mathcal{D}$ and, therefore,
\begin{equation} \label{descompSsumadirecta}
\mathcal{S}:=\mathcal{S}_e\equiv\mathcal{S}_{(0,1)}\oplus
\mathcal{S}_{(1,2)}\oplus...\oplus\mathcal{S}_{(n-1,n)}\oplus\mathcal{S}_{(n,n+1)}
\end{equation}
$\mathcal{S}_{(i,i+1)}=\mathcal{S}\cap\mathfrak{n}_{(i,i+1)}=\mathcal{D}_{(i,i+1)},$
for $0\leq i\leq n.$

\vspace{.5cm} The Euler-d'Alembert-Poincar{\'e} equations are
obtained in the following way:

Let us consider a curve
\begin{equation}\label{descompv}
v(t)=\bigoplus_{i=0}^{n}\eta^{(i,i+1)}(t)\in \mathcal{S},
\end{equation}
such that $\eta^{(i,i+1)}(t)\in\mathcal{S}_{(i,i+1)}.$
Let $\displaystyle\delta
v(t)=\bigoplus_{i=0}^n\delta\eta^{(i,i+1)}(t)\in
\mathfrak{g}$ be an allowed variation, that is, \newline$\delta
v(t)=\dot{\omega}(t)+[v(t),\omega(t)]$ with $\omega(t)\in\mathcal{S}$ such that
$\omega(t_0)=\omega(t_1)=0.$
Then we can write
$\displaystyle\omega(t)=\bigoplus_{j=0}^{n}\xi^{(j,j+1)}(t)$
where
$\xi^{(j,j+1)}(t)\in\mathcal{S}_{(j,j+1)},$ being
$\xi^{(j,j+1)}(t_0)=\varepsilon^{(j,j+1)}(t_1)=0,$ for
$0\leq j\leq n.$

\vspace{.3cm}
Then, as before, we have

\begin{eqnarray*}
\delta v(t) &=& \bigoplus_{i=0}^n\delta\eta^{(i,i+1)}(t)=\dot{\omega}+[v,\omega]=\left(\bigoplus_{j=0}^{n}\dot{\xi}^{(j,j+1)}(t)\right)+
\left[\bigoplus_{i=0}^{n}\eta^{(i,i+1)}(t),
\bigoplus_{j=0}^{n}\xi^{(j,j+1)}(t)\right] \nonumber \\
& & =\bigoplus_{i=0}^{n}\left(\dot{\xi}^{(i,i+1)}(t)+[\eta^{(i,i+1)},
\xi^{(i,i+1)}]+ \sum_{j=0}^{i-1}\left(-a_{(N_j,N_{j+1})}^{(i,i+1)}(\eta^{(j,j+1)},\xi^{(j,j+1)})+ \right.\right.\nonumber \end{eqnarray*}
\begin{eqnarray*}
& & \left.\left.
 b_{(N_j,N_{j+1})}^{(i,i+1)}\left(\eta^{(j,j+1)},\sum_{k=j+1}^{n}
\xi^{(k,k+1)}\right)-b_{(N_j,N_{j+1})}^{(i,i+1)}\left(\xi^{(j,j+1)},\sum_{l=j+1}^{n}
\eta^{(l,l+1)}\right) \right)\right),
\end{eqnarray*}

where $\xi^{(j,j+1)}(t)\in\mathcal{S}_{(j,j+1)},$ for $0\leq j\leq n.$

\vspace{.5cm}
It follows that

\begin{eqnarray}\label{variacEPN}
  \delta\eta^{(i,i+1)}(t) & = & \dot{\xi}^{(i,i+1)}(t)+[\eta^{(i,i+1)},
\xi^{(i,i+1)}]+ \sum_{j=0}^{i-1}\left(-a_{(N_j,N_{j+1})}^{(i,i+1)}(\eta^{(j,j+1)},\xi^{(j,j+1)})+ \right.\\
& & \left.
 b_{(N_j,N_{j+1})}^{(i,i+1)}\left(\eta^{(j,j+1)},\sum_{k=j+1}^n
\xi^{(k,k+1)}\right)-b_{(N_j,N_{j+1})}^{(i,i+1)}\left(\xi^{(j,j+1)},\sum_{l=j+1}^{n}
\eta^{(l,l+1)}\right)\right), \nonumber
\end{eqnarray}

for $\xi^{(j,j+1)}(t)\in\mathcal{S}_{(j,j+1)},$ $0\leq j\leq n.$

\vspace{.4cm}
 If $\displaystyle\mu=\frac{\partial l}{\partial
v}=\bigoplus_{j=0}^{n}\beta_{(j,j+1)}\in\mathfrak{g}^*$
and
$\displaystyle\xi=\bigoplus_{k=0}^{n}\xi^{(k,k+1)}\in\mathfrak{g},$
then the Euler-Poincar{\'e} equations
$\dot{\mu}(\xi)=\langle\mu,[v,\xi]\rangle$ are equivalent to the fact that the equality

\begin{eqnarray}
& & \sum_{j=0}^{n}\dot{\beta}_{(j,j+1)}(\xi^{(j,j+1)})=
\sum_{i=0}^n\left\langle\beta_{(i,i+1)},[\eta^{(i,i+1)},
\xi^{(i,i+1)}]+ \sum_{j=0}^{i-1}\left(-a_{(N_j,N_{j+1})}^{(i,i+1)}(\eta^{(j,j+1)},\xi^{(j,j+1)})+ \right.\right.\nonumber \\
& & \left.\left.
 b_{(N_j,N_{j+1})}^{(i,i+1)}\left(\eta^{(j,j+1)},\sum_{k=j+1}^{n}
\xi^{(k,k+1)}\right)-b_{(N_j,N_{j+1})}^{(i,i+1)}\left(\xi^{(j,j+1)},\sum_{l=j+1}^{n}
\eta^{(l,l+1)}\right) \right)\right\rangle,
\end{eqnarray}

are satisfied for all curves $\xi^{(j,j+1)}(t)\in\mathcal{S}_{(j,j+1)},$ $0\leq j\leq n.$

\vspace{.5cm} Finally,  for this particular case, we obtain the
following \textbf{Euler-d'Alembert-Poincar{\'e} equations by $n$
stages}, which are obtained taking for each $i = 0,...,n$, $\varepsilon^{(i, i+1)}$ arbitrary, while $\varepsilon^{(j, j+1)} = 0$, for $i \neq j$,

\begin{equation}\label{E-Pnohol}
\left\{ \begin{array}{rcl}
\displaystyle \dot{\beta}_{(i,i+1)}\mid_{\mathcal{S}_{(i,i+1)}}&
=&\displaystyle (\text{ad}^*_{\eta^{(i,i+1)
}}\beta_{(i,i+1)})\mid_{\mathcal{S}{(i,i+1)}}+ \beta_{(i,i+1)}
\left.\left(\sum_{j=0}^{i-1}b^{(i,i+1)}_{(N_{j},N_{j+1})}(\eta^{(j,j+1)},\ .\ )\right)\right|_{\mathcal{S}_{(i,i+1)}}  \\
\beta_{(i,i+1)}&=&\displaystyle\frac{\partial l}{\partial \eta^{(i,i+1)}}(v)\\
v&=&\displaystyle
\bigoplus_{i=1}^n\eta^{(i,i+1)}\in\mathcal{S}
\end{array}
\right.
\end{equation}

\vspace{.5cm}\paragraph{General case.} Now suppose that
$\displaystyle\mathcal{D}=\mathcal{S}\subseteq\mathfrak{g}\equiv\mathfrak{n}_{(0,1)}\oplus\mathfrak{n}_{(1,2)}\oplus
...\oplus\mathfrak{n}_{(n-1,n)}\oplus\mathfrak{n}_{(n,n+1)}$ is an
arbitrary subspace.
Let $m=\dim(\mathfrak{g})$ and
$d_i=\dim(\mathfrak{n}_{(i,i+1)}),$ for $0\leq i\leq n.$ Consider a
base $\{\overline{e}_1,...,\overline{e}_m\}$ of $\mathfrak{g}$
adapted to the previous decomposition of $\mathfrak{g}.$ Namely, consider
$\{\overline{e}_1,...,\overline{e}_m\}$ such that the first $d_0$
vectors of the base belong to $\mathfrak{n}_{(0,1)},$ the following $d_1$
vectors belong to $\mathfrak{n}_{(1,2)},$ the following $d_2$ to
$\mathfrak{n}_{(2,3)}$ and so on up to the last $d_n$ vectors
which should belong to $\mathfrak{n}_{(n,n+1)}.$

\vspace{.2cm} Let $(z^1,z^2,...,z^m)$ be the coordinate system in
$\mathfrak{g}$ associated to the base
$\{\overline{e}_1,...,\overline{e}_m\}.$ Then, if $\displaystyle
\xi\in\mathfrak{g},$ we have $\displaystyle \xi=\sum_{i=1}^m z^i
\overline{e}_i.$ In this case we will write
$$\displaystyle
\xi\equiv(z^1,...,z^{d_0},z^{d_0+1},...,z^{d_0+d_1},z^{d_0+d_1+1},...,z^{d_0+d_1+d_2},...,z^{m-d_n},...,z^m).$$

\vspace{.2cm} On the other hand, we can define the arbitrary
subspace $\displaystyle\mathcal{S}\subseteq\mathfrak{g}$ as being the
kernel of a linear function $\displaystyle
f:\mathfrak{g}\rightarrow
\mathbb{R}^{m-\dim(\mathcal{S})}.$
In coordinates we have that points
 $(z^1,z^2,...,z^m)$ belonging to
$\displaystyle\mathcal{S}$
are those which satisfy the system of equations
$f^j(z^1,z^2,...,z^m) = 0,$ for $j = 1,...,m-\dim(\displaystyle\mathcal{S}).$

From this system we can obtain $m-\dim(\displaystyle\mathcal{S})$ coordinates in terms of the rest $\dim(\displaystyle\mathcal{S})$ independent coordinates. More precisely,
we can represent (not necessarily in a unique way) $\mathcal{S}$ as
being the graph of a linear map, as follows. We can choose a set
of indexes $A\subseteq\{1,...,m\},$ with $\dim(\mathcal{S})$
elements, and linear functionals
\begin{equation}\label{coordenS}
\varphi_j(\langle z^i\rangle), \ i\in A,\ j\in\{1,...,m\}-A
\end{equation}
such that points of $\mathcal{S}$ are precisely
$(z^i,\varphi^j(\langle z^i\rangle)),$ with $i\in A$ and
$j\in\{1,...,m\}-A,$ where $z^i$ with $i\in A$ are independent
coordinates.

\vspace{.3cm}
We are going to denote
$\{\varepsilon^0_{r_0}\}=\{\varepsilon_1^0,...,\varepsilon_{s_0}^0\}=\{z^1,...,z^{d_0}\}\cap\{z^i:i\in A\},$ where $r_0 = 1,...,s_0$ and $s_0$ is the dimension of $\mathcal{S}_{(0,1)}$.

The set
$\{h^0_{s_0+1},...,h^0_{d_0-s_0}\}$ is defined by
$\displaystyle \{h^0_{s_0+1},...,h^0_{d_0-s_0}\} = \{z^1,...z^{d_0}\}-\{\varepsilon_1^0,...,\varepsilon_{s_0}^0\}.$

In an analogous way, we define
$\{\varepsilon^1_{r_1}\}=\{z^{d_0+1},...,z^{d_0+d_1}\}\cap\{z^i:i\in
A\},$ for \newline $1\leq r_1\leq s_1$ being $s_1$ the number of
independent variables in $(z^{d_0+1},...,z^{d_0+d_1}).$ Then, we define
$\{h^1_{l_1}\},$ for $s_1+1\leq l_1\leq d_1-s_1,$ by
$\{h^1_{s_1+1},...,h^1_{d_1-s_1}\} = \{z^{d_0+1},...,z^{d_0+d_1}\}-\{\varepsilon^1_1,...,\varepsilon^1_{s_1}\}.$

In the same way,
\newline $\{\varepsilon^2_{r_2}\}=\{z^{d_0+d_1+1},...,z^{d_0+d_1+d_2}\}\cap\{z^i:i\in
A\},$ for $1\leq r_2\leq s_2$ being $s_2$ the number of independent
variables in $(z^{d_0+d_1+1},...,z^{d_0+d_1+d_2}).$ Then, we define
$\{h^2_{l_2}\},$ for $s_2+1\leq l_2\leq d_0+d_1+d_2,$ by
$\{h^2_{s_2+1},...,h^2_{d_0+d_1+d_2}\} = \{z^{d_0+d_1+1},...,z^{d_0+d_1+d_2}\} - \{\varepsilon^2_1,...,\varepsilon^2_{s_2}\}.$

\vspace{.2cm} We continue this procedure to obtain
$\{\varepsilon^i_{r_i}\},$ for $0\leq i\leq n.$ In particular,
\newline $\displaystyle\{\varepsilon^n_{r_n}\}=\{z^{m-d_n},...,z^m\}\cap\{z^i:i\in
A\},$ for $1\leq r_n\leq s_n$ being $s_n$ the number of independent
variables in $(z^{m-d_n},...,z^m).$ Besides, $\{h^n_{l_n}\}$ for
$s_n+1\leq l_n\leq m,$
$\{h^n_{s_n+1},...,h^n_{d_n-s_n}\} = \{z^{m-n},...,z^m\} - \{\varepsilon^n_1,...,\varepsilon^n_{s_n}\}.$

\vspace{.5cm} Then we can write, equivalently, the linear dependence
$\varphi$ as follows:
\begin{eqnarray}
   h^i_{l_i}&=& h^i_{l_i}(\varepsilon^0_{r_0},...,\varepsilon^n_{r_n}),
\end{eqnarray}

where $ 1\leq r_p\leq s_p,$ for $0\leq p \leq n,$ and $s_i+1\leq l_i\leq d_i-s_i,$ for $0\leq i\leq n$.

\vspace{.2cm} Then a generic element
$\xi\equiv(z^1,...,z^m)$
belonging to
$\mathcal{S},$ after a reordering, can be represented as follows
\begin{equation}\label{elemAlgLie}
\xi=(\varepsilon^0_1,...,\varepsilon^0_{s_0},h^0_{s_0+1},...,h^0_{d_0-s_0},\varepsilon^1_1,...,\varepsilon^1_{s_1},h^1_{s_1+1},...,h^1_{d_1-s_1}
,...,\varepsilon^n_1,...,\varepsilon^n_{s_n},h^n_{s_n+1},...,h^n_{d_n-s_n}).
\end{equation}

For simplicity, we shall introduce the following notation.
First, we denote
$(\varepsilon^i_1,...,\varepsilon^i_{s_i}) = \varepsilon^i$ and
$(h^i_{s_i+1},...,h^i_{d_i-s_i}) = h^i,$ for $i = 0,...,n.$
We will define the subspaces $\mathcal{S}_i,$ for $i = 0,...,n,$ as follows
\begin{align}
\mathcal{S}_0
&= \{(\varepsilon^0, h^0(\varepsilon^0, 0,...,0), 0, h^1(\varepsilon^0, 0,...,0),...,0, h^n(\varepsilon^0, 0,...,0))\},\nonumber\\
\mathcal{S}_1
&= \{(0, h^0(0, \varepsilon^1,...,0), \varepsilon^1, h^1(0, \varepsilon^1,...,0),...,0, h^n(0, \varepsilon^1,...,0))\}, \nonumber \\
& ... \nonumber \\
\mathcal{S}_n
&= \{(0, h^0(0,..., \varepsilon^n), 0, h^1(0,..., \varepsilon^n),...,\varepsilon^n, h^n(0,..., \varepsilon^n))\}.\nonumber\\
\end{align}

Then we have that $\mathcal{S}=\mathcal{S}_0\oplus \mathcal{S}_1\oplus ...\oplus \mathcal{S}_n.$
Then we obtain the following decomposition of the Euler-d'Alembert-Poincare equations:

\begin{eqnarray}\label{sistdemus}
\left\{
\begin{array}{rl}
  \dot{\mu}_i
&=\displaystyle\bigoplus_{j,p}\left.\left(\text{ad}^*_{v_j}\mu_p\right)\right|_{\mathcal{S}_i}  \\
\vspace{.2cm}
\mu_i
&=\displaystyle\frac{\partial l}{\partial v_i}(v) \\
& \displaystyle v_j\in\mathcal{S}_j
\end{array}
\right.
\end{eqnarray}

where $i,j,p=0,...,n,$ which is, as usual, a differential-algebraic equation.

Then using the previous decomposition \ref{elemAlgLie} for $\psi$ with $v\equiv\psi,$ in coordinates we have

\begin{eqnarray*}
\mu_0 &=&\frac{\partial l}{\partial v_0}(v)=\left(\frac{\partial l}{\partial z^1},...,\frac{\partial l}{\partial z^{d_0}}\right)(v)=\left(\frac{\partial l}{\partial z^1}(z^1,...,z^m),...,\frac{\partial l}{\partial z^{d_0}}(z^1,...,z^m)\right)\\
 &=&\displaystyle\left(\mu_{0,1},...,\mu_{0,d_0}\right),\\
\mu_1&=&\frac{\partial l}{\partial v_1}(v)=\left(\frac{\partial l}{\partial z^{d_0+1}},...,
\frac{\partial l}{\partial z^{d_0+d_1}}\right)(v)=\left(\frac{\partial l}{\partial z^{d_0+1}}(z^1,...,z^m),...,
\frac{\partial l}{\partial z^{d_0+d_1}}(z^1,...,z^m)\right)\\
& =&\displaystyle\left(\mu_{1,d_0+1},...,\mu_{1,d_0+d_1}\right),\\
&\vdots&\\
\mu_n&=&\frac{\partial l}{\partial v_n}(v)=\left(\frac{\partial l}{\partial z^{m-d_n}},...,
\frac{\partial l}{\partial z^m}\right)(v)=\left(\frac{\partial l}{\partial z^{m-d_n}}(z^1,...,z^m),...,
\frac{\partial l}{\partial z^m}(z^1,...,z^m)\right)\\
& =&\displaystyle\left(\mu_{n,m-d_n},...,\mu_{n,m}\right).\\
\end{eqnarray*}

\vspace{.5cm} Then, the differential-algebraic system (\ref{sistdemus}) can be written as follows,

\begin{eqnarray}\label{E-dA-P-etapas}
\left\{ \begin{array}{ll}
& \displaystyle\sum_{i=0}^{n}\dot{\mu}_i(\varepsilon^i_1,...,\varepsilon^i_{s_i},h^1_{s_i+1}(\varepsilon_r),...,h^i_{d_i-s_i}(\varepsilon_r))=\\
& \displaystyle \sum_{i=0}^n\left\langle\mu_i,[\eta^{(i,i+1)},
(\varepsilon^i_1,...,\varepsilon^i_{s_i},h^1_{s_i+1}(\varepsilon_r),...,h^i_{d_i-s_i}(\varepsilon_r))]+\right. \\
& \left.\displaystyle \sum_{j=0}^{i-1}\left(-a_{(N_j,N_{j+1})}^{(i,i+1)}(\eta^{(j,j+1)},(\varepsilon^j_1,...,
\varepsilon^j_{s_j},h^1_{s_j+1}(\varepsilon_r),...,h^j_{d_j-s_j}(\varepsilon_r)))+ \right.\right.\\
&  \displaystyle\left.\left.
 b_{(N_j,N_{j+1})}^{(i,i+1)}\left(\eta^{(j,j+1)},\sum_{k=j+1}^{n}
(\varepsilon^k_1,...,\varepsilon^k_{s_k},h^1_{s_k+1}(\varepsilon_r),...,h^k_{d_k-s_k}(\varepsilon_r))\right)- \right.\right.\\
& \left.\left.\displaystyle b_{(N_j,N_{j+1})}^{(i,i+1)}\left((\varepsilon^j_1,...,
\varepsilon^j_{s_j},h^1_{s_j+1}(\varepsilon_r),...,h^j_{d_j-s_j}(\varepsilon_r)),\sum_{l=j+1}^{n}
\eta^{(l,l+1)}\right) \right)\right\rangle,\\
&  \displaystyle\mu_i=\displaystyle\frac{\partial l}{\partial v_i}(v).
\end{array}
\right.
\end{eqnarray}

\vspace{.5cm} Here $\varepsilon_r=(\varepsilon^0,...,\varepsilon^n)=(\varepsilon^0_1,...,\varepsilon^0_{s_0},\varepsilon^1_1,...,
\varepsilon^1_{s_1},...,\varepsilon^n_1,...,\varepsilon^n_{s_n})$ are arbitrary, and

$\displaystyle v(t)=\bigoplus_{i=0}^{n}\eta^{(i,i+1)}(t)\in \mathcal{S}$ as in (\ref{descompv}).

\vspace{.3cm} In order to obtain a system of implicit scalar differential equations in terms of
 $\dot{\mu}_{ij},$  $\mu_{ij}$ and $\varepsilon^i_j,$
which will be linear in  $\dot{\mu}_{ij},$
 one may proceed as follows.

\vspace{.2cm}
Taking advantage of the fact that the variables
$\varepsilon^i_{r_i},$  for $ 1\leq r_i\leq s_i,\
0\leq i\leq n,$ are independent, we can choose their values arbitrarily. Then to obtain differential (implicit) equations of motion we only need to fix the value of each one of them to be arbitrary while the rest are fixed to be $0$ in system (\ref{E-dA-P-etapas}). Moreover, since the dependence of the $\varepsilon^i_{r_i}$ is linear we may simply fix the value of each one of them to be $1$ while the rest are fixed to be $0$.

For example, in the system (\ref{E-dA-P-etapas}) we replace
$\varepsilon^0_1=1,\varepsilon^0_2=0...,\varepsilon^0_{s_0}=0,\varepsilon^1_1=0,...,\varepsilon^1_{s_1}=0,\varepsilon^2_1=0,\
...,$ $\varepsilon^2_{s_2}=0,...,
\varepsilon^n_1=0,...,\varepsilon^n_{s_n}=0.$ This gives part of the equations.
Then we replace $\varepsilon^0_1=0,\varepsilon^0_2=1,\  ...,\ \varepsilon^0_{s_0}=0,\ \varepsilon^1_1=0,\ ...,\ \varepsilon^1_{s_1}=0,\
\varepsilon^2_1=0,\ ..., \ \varepsilon^2_{s_2}=0,\ ...,\
\varepsilon^n_1=0,\ ...,\ \varepsilon^n_{s_n}=0,$ which gives some more equations.
We continue in this way until the
last replacement $\varepsilon^0_1=0,\varepsilon^0_2=0,\  ...,\ \varepsilon^0_{s_0}=0,\
\varepsilon^1_1=0,\ ...,\ \varepsilon^1_{s_1}=0,\
\varepsilon^2_1=0,\ ...,\ \varepsilon^2_{s_2}=0,\ ..., \
\varepsilon^n_1=0,\ ...,\ \varepsilon^n_{s_n}=1.$\\

To the system of implicit differential equations so obtained we should add the algebraic restrictions
$\partial l / \partial v_i = \mu_i.$ This gives the Euler-d'Alembert-Poincare equations in coordinates.\\

\textbf{Remark.}
The particular case
$\displaystyle\mathcal{S}\equiv\mathcal{S}_{(0,1)}\oplus
\mathcal{S}_{(1,2)}\oplus...\oplus\mathcal{S}_{(n-1,n)}\oplus\mathcal{S}_{(n,n+1)}$
corresponds to the case in which all the variables are independent, that is, there are no variables $h^i_j.$

\vspace{.2cm} \section{Local Lagrange-Poincar{\'e} equations by several stages.}\label{sect:Lag-Poin}

\vspace{.5cm}
Let us first calculate the local Lagrange-Poincar{\'e} equations.
 Let $G$ be a Lie group and $Q=X\times G$ a trivial principal bundle with structure group $G$
acting on the left by  $g(x,h):=(x,gh),\ \forall\ g,h\in G,\ \forall x\in X,$
where the projection $\pi_G:X\times G\rightarrow X$ is given by $\pi_G (x,g)=x,$
for all $(x,g)\in Q.$
Let a principal connection
$\mathcal{A}:TQ\rightarrow\mathfrak{g},$ then
\begin{eqnarray*}
\mathcal{A}(x,g,\dot{x},\dot{g})&=&\mathcal{A}(x,g,\dot{x},0)+\mathcal{A}(x,g,0,\dot{g})\\
& =& g\mathcal{A}(x,e,\dot{x},0)g^{-1}+
g\mathcal{A}(x,e,0,g^{-1}\dot{g})g^{-1}=g\mathcal{A}(x)\dot{x}g^{-1}+\dot{g}g^{-1},
\end{eqnarray*}
where $\mathcal{A}(x)\dot{x}=\mathcal{A}(x,e,\dot{x},0)$ by
definition.
Then we can deduce that $(x,e,\dot{x},\xi)$ belongs to $T_{(x,e)}(X\times
G)$ is a horizontal vector if and only if
$\mathcal{A}(x)\dot{x}+\xi=0.$

\vspace{.3cm} On the other hand, the Cartan's structure equation says that the curvaure $B$ is given by
$B(X,Y)=d\mathcal{A}(X,Y)-[\mathcal{A}(X),\mathcal{A}(Y)],$ $
\forall \ X,Y\in\mathfrak{X}(Q).$

\vspace{.3cm} If $X$ and $Y$ are horizontal fields, then
$B(X,Y)=-\mathcal{A}([X,Y]).$

\vspace{.3cm} According to Appendix \ref{The Geometry of the Reduced Bundles} we have an isomorphism
$\alpha_{\mathcal{A}}:TQ/G\rightarrow
T(Q/G)\oplus\widetilde{\mathfrak{g}}$ given by
$\alpha_{\mathcal{A}}([q,\dot{q}]_G)=(x,\dot{x})\oplus[q,\mathcal{A}(q,\dot{q})]_G,$
where $(x,\dot{x})=T\pi(q,\dot{q}),$
$\widetilde{\mathfrak{g}}=(Q\times\mathfrak{g})/G$ with the action of
$G$ over $Q \times \mathfrak{g}$ given by
$g(q,\xi)=(gq,\text{Ad}_g\xi).$

\vspace{.3cm} In the local case $Q=X\times G,$ we have
$\widetilde{\mathfrak{g}}\equiv X\times\mathfrak{g}$ through the
identification
\newline $[(x,g), \xi]_G\equiv(x,\text{Ad}_{g^{-1}}\xi),$ in particular we have
$[(x,e),\xi]_G\equiv(x,\xi).$
Then the isomorphism $\alpha_{\mathcal{A}}$
is given by

$$\alpha_{\mathcal{A}}([x,g,\dot{x},\dot{g}]_G)=(x,\dot{x})\oplus(\mathcal{A}(x)\dot{x}+g^{-1}\dot{g}).$$

\vspace{.3cm} Now let us calculate the $\tilde{\mathfrak{g}}$-valued two form $\tilde{B}.$ By definition,

$$\widetilde{B}(x)(\dot{x},\delta x)=[q,B(q)(\dot{x}_q^h,\delta
x_q^h)]_G$$

being $\dot{x}_q^h,\delta x_q^h$ the horizontal lifts of
$\dot{x}$ and $\delta x$ at the point $q\in\pi^{-1}(x).$ Of course,
$\dot{x}_q^h,\delta x_q^h$ can be replaced by any pair of elements
$\dot{q},\delta q$ that satisfies the conditions
$T\pi\dot{q}=\dot{x}$ and $T\pi\delta q=\delta x.$

\vspace{.3cm} Since $Q=X\times G,$ we have
\begin{eqnarray*}
   \widetilde{B}(x)(\dot{x},\delta x)&=& [(x,e),B(x,e)((\dot{x},0),(\delta x,0))]_G=(x,B(x,e)((\dot{x},0),(\delta x,0))).
\end{eqnarray*}

Using the Cartan's structure equation, we obtain

$\displaystyle B(x,e)((\dot{x},0),(\delta x,0))=d\mathcal{A}(x)(\dot{x},\delta
x)-[\mathcal{A}(x)\dot{x},\mathcal{A}(x)\delta x],$ then
$$\widetilde{B}(x)(\dot{x},\delta
x)=(x,d\mathcal{A}(x)(\dot{x},\delta
x)-[\mathcal{A}(x)\dot{x},\mathcal{A}(x)\delta x]).$$

 Now we need to calculate the covariant derivatives on the vector bundles
$\widetilde{\mathfrak{g}}$ and $\widetilde{\mathfrak{g}}^*:$

\vspace{.3cm} We have the formula
$\displaystyle\frac{D}{Dt}[q,\xi]_G=[q,-[\mathcal{A}(q,\dot{q}),\xi]+\dot{\xi}]_G$
(see Appendix \ref{appendixa}).

\vspace{.2cm} In the local case, after the identification
$[(x,e),\xi]_G\equiv(x,\xi),$ we have
\begin{eqnarray}\label{dercovtrivial}
 \frac{D}{Dt}(x,\xi)= (x,-[\mathcal{A}(x)\dot{x},\xi]+\dot{\xi})=(x,\dot{\xi}-\text{ad}_{\mathcal{A}(x)\dot{x}}\ \xi).
\end{eqnarray}

\vspace{.3cm} Let $\overline{\alpha}(t)=(x(t),\alpha(t))\in
\widetilde{\mathfrak{g}}^*$ with the identification
$\widetilde{\mathfrak{g}}^*\equiv X\times \mathfrak{g}^*$ given by

$[(x,e),\alpha]_G\equiv(x,\alpha).$

\vspace{.3cm} For any curve
$\overline{\xi}(t)=(x(t),\xi(t))\in\widetilde{\mathfrak{g}},$ we have
\begin{equation*}
   \frac{d}{dt}\langle\overline{\alpha}(t),\overline{\xi}(t)\rangle
   = \dot{\alpha}(t)\xi(t)+\alpha(t)\dot{\xi}(t) = \left\langle\frac{D\overline{\alpha}}{Dt},\overline{\xi}(t)\right\rangle+\left\langle\overline{\alpha}(t),
\frac{D\overline{\xi}(t)}{Dt}\right\rangle.
\end{equation*}

But using the identity (\ref{dercovtrivial}), we obtain that
$$\left\langle\overline{\alpha}(t),\frac{D\overline{\xi}(t)}{Dt}\right\rangle=\left\langle\alpha(t),-[\mathcal{A}(x)\dot{x},
\xi(t)]+\dot{\xi}(t)\right\rangle=\left\langle\alpha(t),\dot{\xi}(t)-\text{ad}_{\mathcal{A}(x)\dot{x}}\
\xi(t)\right\rangle.$$

Then
\begin{eqnarray}
   \frac{D(x,\alpha)}{Dt}=(x,\dot{\alpha}+\text{ad}^*_{\mathcal{A}(x)\dot{x}}\ \alpha).
\end{eqnarray}

\paragraph{The local vertical Lagrange-Poincar{\'e} equation.}

As in Appendix \ref{reducedLPE} for a reduced Lagrangian $l:T(Q/G)\oplus\widetilde{\mathfrak{g}}\rightarrow
\mathbb{R},$ we have the vertical Lagrange-Poincar\'e equations
\begin{eqnarray}
   \frac{D}{Dt}\frac{\partial l}{\partial \overline{v}}=\text{ad}^*_{\overline{v}}\ \frac{\partial l}{\partial \overline{v}}.
\end{eqnarray}

\vspace{.2cm} In the local case, we denote
$l(x,\dot{x},[e,\xi]_G)\equiv l(x,\dot{x},(x,\xi))\equiv
l(x,\dot{x},\xi)$ and we can identify $\partial
l/\partial\overline{v}\equiv\partial l/\partial \xi.$ Then,
using the formula (\ref{dercovtrivial}), we obtain
$$\frac{D}{Dt}\frac{\partial l}{\partial\overline{v}}\equiv\frac{D}{Dt}
\frac{\partial l}{\partial \xi}=\frac{d}{dt}\frac{\partial l}{\partial\xi} +\text{ad}^*_{\mathcal{A}(x)\dot{x}}\frac{\partial l}{\partial \xi}.$$

\vspace{.2cm} Then the vertical Lagrange-Poincar{\'e} equation becomes
\begin{eqnarray}
   \frac{d}{dt}\frac{\partial l}{\partial\xi}=\text{ad}^*_{\xi}
   \frac{\partial l}{\partial\xi}-\text{ad}^*_{\mathcal{A}(x)\dot{x}}\frac{\partial l}{\partial\xi}.
\end{eqnarray}

\vspace{.3cm}
\paragraph{The local horizontal Lagrange-Poincar{\'e} equations.}

In the general case we have
\begin{eqnarray}
   \frac{\partial^{\mathcal{A}}l}{\partial x}\ \delta x-\frac{D}{Dt}\frac{\partial l}{\partial\dot{x}}\ \delta{x}=\frac{\partial l}{\partial \overline{v}}\widetilde{B}(x)(\dot{x},\delta x).
\end{eqnarray} Now we need to calculate this in the local case. For this we proceed as in Appendix \ref{reducedLPE}.

\vspace{.2cm} Now let $\overline{v}(t,s)$ be a horizontal
deformation of the curve $\overline{v}(t)$ such that
$\overline{v}(t,0)=\overline{v}(t),$ and let $x(t,s)$ be the base
point of $\overline{v}(t,s),$ then $x(t,0)=x(t).$

In the local case, we have $\overline{v}(t,s)=(x(t,s),\xi(t,s)).$
Let $\delta x=\displaystyle\left.\frac{\partial x}{\partial s}\right|_{s=0}.$

By definition, $$\displaystyle
\frac{\partial^{\mathcal{A}}l}{\partial x}\ \delta
x=\left.\frac{d}{ds}\right|_{s=0}\frac{dl}{dx}(x(t,s),\dot{x}(t,s),(x,\xi)(t,s)).$$

\vspace{.3cm} \underline{Note}: $\dot{x}(t)$ has been constant.

\vspace{.3cm} Since $(x(t,s),\xi(t,s))$ is horizontal for each $t,$
 using the formula (\ref{dercovtrivial}) we have
$$\left.\frac{\partial\xi}{\partial s}\right|_{s=0}=[\mathcal{A}(x)\delta x,\xi]\equiv\text{Ad}_{\mathcal{A}(x)\delta x}\xi$$

then, $$\frac{\partial^{\mathcal{A}}l}{\partial x}\ \delta
x=\frac{\partial l}{\partial x}\ \delta x+\frac{\partial l}{\partial\xi}[\mathcal{A}(x)\delta x,\xi]\equiv
\frac{\partial l}{\partial x}\ \delta x+\frac{\partial l}{\partial\xi}\text{ad}_{\mathcal{A}(x)\delta x}\xi.$$

\vspace{.3cm} Finally the horizontal Lagrange-Poincar{\'e} equations
are

\begin{eqnarray}
   \frac{\partial l}{\partial x}\ \delta x-\frac{d}{dt}\frac{\partial l}{\partial\dot{x}}\ \delta x=
   \frac{\partial l}{\partial \xi}\widetilde{B}(x)(\dot{x},\delta x)-\frac{\partial l}{\partial\xi}[\mathcal{A}(x)\delta x,\xi].
\end{eqnarray}

\vspace{.5cm} So, the Lagrange-Poincar{\'e} system of vertical and
horizontal equations is

\begin{eqnarray}\label{L-Ptrivial}
  \left \{ \begin{array}{l}
              \displaystyle\frac{d}{dt}\frac{\partial l}{\partial\xi} =\text{ad}^*_{\xi}
              \frac{\partial l}{\partial\xi}-\text{ad}^*_{\mathcal{A}(x)\dot{x}}\frac{\partial l}{\partial \xi} \\
              \\
              \displaystyle\frac{\partial l}{\partial x}\ \delta x-\frac{d}{dt}\frac{\partial l}{\partial\dot{x}}\ \delta x  =\frac{\partial l}{\partial\xi}\widetilde{B}(x)(\dot{x},\delta x)-
              \frac{\partial l}{ \partial \xi}[\mathcal{A}(x)\delta x,\xi]
                 \end{array} \right.
\end{eqnarray}

\paragraph{Local Lagrange-Poincar{\'e} equations by several stages.}

\vspace{.5cm} Consider $Q=X\times G$ and the identification
$\widetilde{\mathfrak{g}}\equiv X\times \mathfrak{g}$ given by
$[(x,e),\xi]_G\equiv(x,\xi)$ and
$[(x,g),\xi]_G\equiv(x,e,\text{Ad}_{g^{-1}}\xi).$

We have $\mathfrak{g}\equiv\mathfrak{n}_{(0,1)}\oplus\mathfrak{n}_{(1,2)}\oplus
...\oplus\mathfrak{n}_{(n,n+1)}$ and
$\displaystyle\mathfrak{g}^*=\bigoplus_{i=0}^n\mathfrak{n}_{(i,i+1)}^*,$
where we are using the same notation as in previous sections.

Let $\xi\in\mathfrak{g},$ namely
\begin{eqnarray}\label{descomppsi}
\displaystyle\xi=\bigoplus_{i=0}^n\eta^{(i,i+1)},
\end{eqnarray}
then $\partial l/\partial\overline{v}\equiv\partial l/\partial\xi\in\mathfrak{g}^*$
and
\begin{eqnarray}\label{descompderivresppsi}
\displaystyle\frac{\partial l}{\partial\xi}=\bigoplus_{j=0}^n
\frac{\partial l}{\partial\eta^{(j,j+1)}}=\bigoplus_{j=0}^n\beta_{(j,j+1)},
\end{eqnarray}
where $\beta_{(j,j+1)}\in\mathfrak{n}^*_{(j,j+1)},$ for $0\leq j\leq n.$

\vspace{.3cm} On the other hand,
$\mathcal{A}(x)\dot{x}\in\mathfrak{g}$ and, so, it can be decomposed
as $\displaystyle\mathcal{A}(x)\dot{x}=\bigoplus_{i=0}^n\mathcal{A}^{(i,i+1)}
(x,\dot{x})$ with $\mathcal{A}^{(i,i+1)}(x,\dot{x})\in\mathfrak{n}_{(i,i+1)},$ $0\leq i\leq n.$

\paragraph{Local vertical Lagrange-Poincar{\'e} equations by several stages}

For a given arbitrary element
$\displaystyle\rho=\bigoplus_{i=0}^n\nu^{(i,i+1)}\in\mathfrak{g}$
the vertical Lagrange-Poincar{\'e} equations
(\ref{L-Ptrivial}) at this point $\rho$ are:

\begin{eqnarray*}
\left\langle\bigoplus_{j=0}^n\dot{\beta}_{(j,j+1)},\bigoplus_{i=0}^n
\nu^{(i,i+1)}\right\rangle &=& \left\langle\bigoplus_{l=0}^n\beta_{(l,l+1)},\left[
\bigoplus_{j=0}^n\eta^{(j,j+1)},\bigoplus_{i=0}^n\nu^{(i,i+1)}\right]\right
\rangle\\
& & -\left\langle\bigoplus_{l=0}^n\beta_{(l,l+1)},\left[
\bigoplus_{k=0}^n\mathcal{A}^{(k,k+1)}(x,\dot{x}),\bigoplus_{j=0}^n\nu^{(j,j+1)}\right]\right\rangle.
\end{eqnarray*}

\vspace{.5cm} So, we obtain

$$\bigoplus_{j=0}^n\dot{\beta}_{(j,j+1)}(\nu^{(j,j+1)})=\left\langle\bigoplus_{l=0}^n\beta_{(l,l+1)},
\bigoplus_{i=0}^{n}\left([\eta^{(i,i+1)},
\nu^{(i,i+1)}]-[\mathcal{A}^{(i,i+1)}(x,\dot{x}),
\nu^{(i,i+1)}]+ \right.\right. \nonumber $$
$$ \left.\left. \sum_{j=0}^{i-1}\left(-a_{(N_j,N_{j+1})}^{(i,i+1)}(\eta^{(j,j+1)},\nu^{(j,j+1)})+ b_{(N_j,N_{j+1})}^{(i,i+1)}\left(\eta^{(j,j+1)},\sum_{k=j+1}^{n}
\nu^{(k,k+1)}\right)- \right.\right.\right.$$
$$ \left.\left.\left. b_{(N_j,N_{j+1})}^{(i,i+1)}\left(\nu^{(j,j+1)},\sum_{l=j+1}^{n}
\eta^{(l,l+1)} \right)+ a_{(N_j,N_{j+1})}^{(i,i+1)}(\mathcal{A}^{(j,j+1)}(x,\dot{x}),\nu^{(j,j+1)})\right. \right.\right.
\nonumber $$
$$\left. \left.\left.-
 b_{(N_j,N_{j+1})}^{(i,i+1)}\left(\mathcal{A}^{(j,j+1)}(x,\dot{x}),\sum_{k=j+1}^{n}
\nu^{(k,k+1)}\right)+ b_{(N_j,N_{j+1})}^{(i,i+1)}\left(\nu^{(j,j+1)},\sum_{l=j+1}^{n}
\mathcal{A}^{(l,l+1)}(x,\dot{x})\right)\right)\right)\right\rangle.$$

\vspace{.5cm} Then we have

$$\sum_{j=0}^n\dot{\beta}_{(j,j+1)}(\nu^{(j,j+1)})=\sum_{i=0}^n\left\langle\beta_{(i,i+1)},
[\eta^{(i,i+1)},
\nu^{(i,i+1)}]-[\mathcal{A}^{(i,i+1)}(x,\dot{x}),
\nu^{(i,i+1)}]+ \right. \nonumber $$
$$ \left. \sum_{j=0}^{i-1}\left(-a_{(N_j,N_{j+1})}^{(i,i+1)}(\eta^{(j,j+1)},\nu^{(j,j+1)})+ b_{(N_j,N_{j+1})}^{(i,i+1)}\left(\eta^{(j,j+1)},\sum_{k=j+1}^{n}
\nu^{(k,k+1)}\right)- \right.\right.$$
$$ \left.\left. b_{(N_j,N_{j+1})}^{(i,i+1)}\left(\nu^{(j,j+1)},\sum_{l=j+1}^{n}
\eta^{(l,l+1)} \right)+ a_{(N_j,N_{j+1})}^{(i,i+1)}(\mathcal{A}^{(j,j+1)}(x,\dot{x}),\nu^{(j,j+1)})\right. \right.
\nonumber $$
$$\left.\left.-
 b_{(N_j,N_{j+1})}^{(i,i+1)}\left(\mathcal{A}^{(j,j+1)}(x,\dot{x}),\sum_{k=j+1}^{n}
\nu^{(k,k+1)}\right)+ b_{(N_j,N_{j+1})}^{(i,i+1)}\left(\nu^{(j,j+1)},\sum_{l=j+1}^{n}
\mathcal{A}^{(l,l+1)}(x,\dot{x})\right)\right)\right\rangle.$$

\vspace{1cm} Since $\nu^{(i,i+1)}\in\mathfrak{n}_{(i,i+1)}$ for $0\leq i\leq n,$ are
arbitrary, we have the vertical Lagrange-Poincar{\'e} equations by
stages:

\begin{equation*}
\displaystyle\dot{\beta}_{(i,i+1)}\mid_{\mathfrak{n}_{(i,i+1)}}=\left.\left\langle\beta_{(i,i+1)},
[\eta^{(i,i+1)}-\mathcal{A}^{(i,i+1)}(x,\dot{x}),\ .\ ]+ \sum_{j=0}^{i-1}
b^{(i,i+1)}_{(N_j,N_{j+1})}(\eta^{(j,j+1)}-\mathcal{A}^{(j,j+1)}(x,\dot{x}),\ .\ )\right\rangle\right|_{\mathfrak{n}_{(i,i+1)}},
\end{equation*}

\vspace{.3cm} being
$\displaystyle\eta^{(i,i+1)}\in \mathfrak{n}_{(i,i+1)},$ $\ \displaystyle\beta_{(i,i+1)}=\frac{\partial
l}{\partial\eta^{(i,i+1)}}\in \mathfrak{n}^*_{(i,i+1)},$ for $0\leq i\leq n ;$ \quad and

$\displaystyle\mathcal{A}(x)\dot{x}=\bigoplus_{k=0}^n\mathcal{A}^{(k,k+1)}(x,\dot{x}),$ with $\displaystyle\mathcal{A}^{(k,k+1)}\in \mathfrak{n}_{(k,k+1)},$ for $0\leq k \leq n.$

\vspace{.5cm} This equations coincide with the Euler-Poincar{\'e}
equations (\ref{E-Pstages2}) if $Q = G$.

\vspace{.3cm} In the case in which the number of stages is two, these vertical equations by stages are:

\begin{eqnarray*}
   \displaystyle\left.\dot{\beta}_{(0,1)}\right|_{\mathfrak{n}_{(0,1)}} &=&\left.\displaystyle\left\langle\beta_{(0,1)},[\eta^{(0,1)}-\mathcal{A}^{(0,1)}(x,\dot{x}),\ .\ ]\right\rangle\right|_{\mathfrak{n}_{(0,1)}}=\left.\left(\text{ad}^*_{\eta^{(0,1)}-\mathcal{A}^{(0,1)}(x.\dot{x})}
   \beta_{(0,1)}\right)\right|_{\mathfrak{n}_{(0,1)}}\\
   & & \\
   \displaystyle\left.\dot{\beta}_{(1,2)}\right|_{\mathfrak{n}_{(1,2)}} &=&\displaystyle\left.\left\langle\beta_{(1,2)},[\eta^{(1,2)}-\mathcal{A}^{(1,2)}(x,\dot{x}),\ .\ ]+b_{(N_0,N_1)}^{(1,2)}(\eta^{(0,1)}-\mathcal{A}^{(0,1)}(x,\dot{x}),\ .\ )\right\rangle\right|_{\mathfrak{n}_{(1,2)}}\\
   &= & \displaystyle\left.\left(\text{ad}^*_{\eta^{(1,2)}-\mathcal{A}^{(1,2)}(x,\dot{x})}
   \beta_{(1,2)}\right)\right|_{\mathfrak{n}_{(1,2)}}+\left.\beta_{(1,2)}(b^{(1,2)}_{(N_0,N_1)}(\eta^{(0,1)}-
   \mathcal{A}^{(0,1)}(x,\dot{x}),\ .\ ))\right|_{\mathfrak{n}_{(1,2)}}\\
   & & \\
   \displaystyle\left.\dot{\beta}_{(2,3)}\right|_{\mathfrak{n}_{(2,3)}} &=&\displaystyle\left.\left\langle\beta_{(2,3)},[\eta^{(2,3)}-\mathcal{A}^{(2,3)}(x,\dot{x}),\ .\ ]+b_{(N_0,N_1)}^{(2,3)}(\eta^{(0,1)}-\mathcal{A}^{(0,1)}(x,\dot{x}),\ .\ )+\right.\right.\\
   & & \displaystyle \left.\left.
   b_{(N_1,N_2)}^{(2,3)}(\eta^{(1,2)}-\mathcal{A}^{(1,2)}(x,\dot{x}),\ .\ )
   \right\rangle\right|_{\mathfrak{n}_{(2,3)}}=\\
   & & \displaystyle
   \hspace{-2cm}\left.\left(\text{ad}^*_{\eta^{(2,3)}-\mathcal{A}^{(2,3)}(x,\dot{x})}
   \beta_{(2,3)}\right)\right|_{\mathfrak{n}_{(2,3)}}+\left.\beta_{(2,3)}(b^{(2,3)}_{(N_0,N_1)}(\eta^{(0,1)}-
   \mathcal{A}^{(0,1)}(x,\dot{x})+\eta^{(1,2)}-\mathcal{A}^{(1,2)}(x,\dot{x}),\ .\ ))\right|_{\mathfrak{n}_{(2,3)}}\\
\end{eqnarray*}

\paragraph{The local horizontal Lagrange-Poincar{\'e} equations by several stages.}

\vspace{.5cm} Recall  that the horizontal Lagrange-Poincar{\'e} equations are

\begin{eqnarray}\label{horizE-P}
\frac{\partial l}{\partial x}\delta x-\frac{d}{dt}\frac{\partial
l}{\partial\dot{x}}\delta x=\frac{\partial
l}{\partial\xi}\widetilde{B}(x)(\dot{x},\delta x)-\frac{\partial
l}{\partial\xi}[\mathcal{A}(x)\delta x,\xi].
\end{eqnarray}

\vspace{.2cm} Let $\xi\in\mathfrak{g}$ such that
$\displaystyle\xi=\bigoplus_{0=1}^n\eta^{(i,i+1)}$
and $\displaystyle\frac{\partial
l}{\partial\xi}=\bigoplus_{j=0}^n\beta_{(j,j+1)},$
as in (\ref{descomppsi}) and (\ref{descompderivresppsi}). Let
$\mathcal{A}(x)\delta x\in\mathfrak{g}$ such that
$\displaystyle\mathcal{A}(x)\delta x=\bigoplus_{j=0}^n\mathcal{A}^{(j,j+1)}(x,\delta
x).$ Also consider $\displaystyle B(x,e)((\dot{x},0),(\delta
x,0))=\bigoplus_{k=0}^n\varphi^{(k,k+1)}$, where $\varphi^{(k,k+1)} \in \mathfrak{n}_{(k,k+1)}$, $k = 0,...,n$.

\vspace{.5cm} Then, since $\widetilde{B}(x)(\dot{x},\delta
x)=(x,B(x,e)((\dot{x},0),(\delta x,0)))\equiv
B(x,e)((\dot{x},0),(\delta x,0)),$ we can write the right hand side of
the equation (\ref{horizE-P}) as

$$\left\langle\bigoplus_{l=0}^n\beta_{(l,l+1)},\bigoplus_{k=0}^n
\varphi^{(k,k+1)}-\left[\bigoplus_{i=0}^n\mathcal{A}^{(i,i+1)}(x,\delta
x),\bigoplus_{j=0}^n\eta^{(j,j+1)}\right]\right\rangle=\left\langle\bigoplus_{l=0}^n\beta_{(l,l+1)},
\left(\bigoplus_{s=0}^n
\varphi^{(s,s+1)}\right)-\right.$$
$$\left.
\bigoplus_{i=0}^{n}\left([\mathcal{A}^{(i,i+1)}(x,\delta x),
\eta^{(i,i+1)}]+\sum_{j=0}^{i-1}\left(-a_{(N_j,N_{j+1})}^{(i,i+1)}(\mathcal{A}^{(j,j+1)}(x,\delta x),\eta^{(j,j+1)})
\right.\right.\right.$$
$$\left.\left.\left.+
 b_{(N_j,N_{j+1})}^{(i,i+1)}\left(\mathcal{A}^{(j,j+1)}(x,\delta x),\sum_{k=j+1}^{n}
\eta^{(k,k+1)}\right)
-b_{(N_j,N_{j+1})}^{(i,i+1)}\left(\eta^{(j,j+1)},\sum_{l=j+1}^{n}
\mathcal{A}^{(l,l+1)}(x,\delta x)\right) \right)\right)\right\rangle.$$

\vspace{.5cm} Then we have the horizontal Lagrange-Poincar{\'e}
equations by stages:

$$\frac{\partial l}{\partial x}\ \delta x-\frac{d}{dt}\frac{\partial l}{\partial\dot{x}}\ \delta x=
\sum_{i=0}^n\left\langle\beta_{(i,i+1)},
\varphi^{(i,i+1)}-[\mathcal{A}^{(i,i+1)}(x,\delta x),\eta^{(i,i+1)}]-\right.$$
$$\left.
\sum_{j=0}^{i-1}\left(-a_{(N_j,N_{j+1})}^{(i,i+1)}(\mathcal{A}^{(j,j+1)}(x,\delta x),\eta^{(j,j+1)})+
 b_{(N_j,N_{j+1})}^{(i,i+1)}\left(\mathcal{A}^{(j,j+1)}(x,\delta x),\sum_{k=j+1}^{n}
\eta^{(k,k+1)}\right)\right.\right.$$
$$\left.\left.
-b_{(N_j,N_{j+1})}^{(i,i+1)}\left(\eta^{(j,j+1)},\sum_{l=j+1}^{n}
\mathcal{A}^{(l,l+1)}(x,\delta x)\right) \right)\right\rangle;$$

\vspace{.5cm} being $\displaystyle\beta_{(i,i+1)}=\frac{\partial
l}{\partial\eta^{(i,i+1)}}\in\mathfrak{n}^*_{(i,i+1)},$ for $0\leq i\leq n.$

\paragraph{The case of two stages.} In the particular case in which we have two stages,
the horizontal equations are:

$$\frac{\partial l}{\partial x}\ \delta x-\frac{d}{dt}\frac{\partial l}{\partial\dot{x}}\ \delta x=
\langle\beta_{(0,1)},\varphi^{(0,1)}-[\mathcal{A}^{(0,1)}(x,\delta x),\eta^{(0,1)}]\rangle+
\left\langle\beta_{(1,2)},\varphi^{(1,2)}-[\mathcal{A}^{(1,2)}(x,\delta x),\eta^{(1,2)}]\right.$$
$$\left.-\left(-a^{(1,2)}_{(N_0,N_1)}(\mathcal{A}^{(0,1)}(x,\delta x),\eta^{(0,1)})+b^{(1,2)}_{(N_0,N_1)}(\mathcal{A}^{(0,1)}(x,\delta x),\eta^{(1,2)}+\eta^{(2,3)})\right.\right.$$ $$\left.\left.-b^{(1,2)}_{(N_0,N_1)}(\eta^{(0,1)},\mathcal{A}^{(1,2)}(x,\delta
x)+\mathcal{A}^{(2,3)}(x,\delta x))\right)\right\rangle + \left\langle\beta_{(2,3)},\varphi^{(2,3)}-[\mathcal{A}^{(2,3)}(x,\delta x),\eta^{(2,3)}]\right. $$
$$\left. -\left(-a^{(2,3)}_{(N_0,N_1)}(\mathcal{A}^{(0,1)}(x,\delta x),\eta^{(0,1)})+b^{(2,3)}_{(N_0,N_1)}(\mathcal{A}^{(0,1)}(x,\delta x),\eta^{(1,2)}+\eta^{(2,3)})-b^{(2,3)}_{(N_0,N_1)}(\eta^{(0,1)},\mathcal{A}^{(2,3)}(x,\delta x))\right)-\right.$$
$$\left.\left(-a^{(2,3)}_{(N_1,N_2)}(\mathcal{A}^{(1,2)}(x,\delta x),\eta^{(1,2)})+b^{(2,3)}_{(N_1,N_2)}(\mathcal{A}^{(1,2)}(x,\delta x),\eta^{(2,3)})-b^{(2,3)}_{(N_1,N_2)}(\eta^{(1,2)},\mathcal{A}^{(2,3)}(x,\delta x))\right)\right\rangle.$$

\vspace{.5cm} \section{Local Lagrange-d'Alembert-Poincar{\'e}
equations by several stages} \label{sect:Lag-d'Alem-Poin}

\vspace{.5cm} Now we consider the following situation:

\vspace{.2cm} Let $Q=X\times G$ be a principal trivial bundle with
structure group $G.$ Let $G$ in the condition of the theorem
\ref{formcorchete} in the page \pageref{formcorchete}. As before,
we have the identification of the adjoint bundle
$\widetilde{\mathfrak{g}}\equiv X\times\mathfrak{g}.$
We shall consider first a particular case and then the general case, as we did in section \ref{EDPBSS}.
\vspace{.1cm}

\vspace{.2cm} \paragraph{Particular case.} Suppose, as before, that
$\mathcal{S}$ can be decomposed as

\vspace{.2cm}
$\displaystyle\mathcal{S}:=\mathcal{S}_e\equiv\mathcal{S}_{(0,1)}\oplus\mathcal{S}_{(1,2)}\oplus...\oplus\mathcal{S}_{(n,n+1)}$
being $\mathcal{S}_{(i,i+1)}=\mathcal{S}\cap\mathfrak{n}_{(i,i+1)},$ for
$0\leq i\leq n.$

\vspace{.5cm} Then, the
Lagrange-d'Alembert-Poincar{\'e} equations in several stages
are:

\begin{eqnarray}
\displaystyle\dot{\beta}_{(i,i+1)}(\nu^{(i,i+1)})&=&\left\langle\beta_{(i,i+1)},
[\eta^{(i,i+1)}-\mathcal{A}^{(i,i+1)}(x,\dot{x}),\nu^{(i,i+1)}]+ \right.\nonumber \\
& & \left.\sum_{j=0}^{i-1}
b^{(i,i+1)}_{(N_j,N_{j+1})}(\eta^{(j,j+1)}-\mathcal{A}^{(j,j+1)}(x,\dot{x}),\nu^{(i,i+1)})\right\rangle, \nonumber \\
\displaystyle\frac{\partial l}{\partial x}\ \delta x-\frac{d}{dt}\frac{\partial l}{\partial\dot{x}}\ \delta x &=&\sum_{i=0}^n\left\langle\beta_{(i,i+1)},
\varphi^{(i,i+1)}-[\mathcal{A}^{(i,i+1)}(x,\delta x),\eta^{(i,i+1)}]-\right.\nonumber \\
& & \left.
\sum_{j=0}^{i-1}\left(-a_{(N_j,N_{j+1})}^{(i,i+1)}(\mathcal{A}^{(j,j+1)}(x,\delta x),\eta^{(j,j+1)})+\right.\right.\nonumber\\
& & \left.\left.
 b_{(N_j,N_{j+1})}^{(i,i+1)}\left(\mathcal{A}^{(j,j+1)}(x,\delta x),\sum_{k=j+1}^{n}
\eta^{(k,k+1)}\right)\right.\right.\nonumber \\
& & \left.\left.
-b_{(N_j,N_{j+1})}^{(i,i+1)}\left(\eta^{(j,j+1)},\sum_{l=j+1}^{n}
\mathcal{A}^{(l,l+1)}(x,\delta x)\right) \right)\right\rangle;\nonumber
\end{eqnarray}

\vspace{.5cm} for all $\delta x\in T_x(Q/G)$ and $\nu^{(i,i+1)}\in\mathcal{S}_{(i,i+1)},$ $i=0,...,n;$ being
$\displaystyle\eta^{(i,i+1)}\in\mathfrak{n}_{(i,i+1)},$

$\ \displaystyle\beta_{(i,i+1)}=\frac{\partial
l}{\partial\eta^{(i,i+1)}}\in\mathfrak{n}^*_{(i,i+1)}$ and
$\displaystyle\mathcal{A}(x)\dot{x}=\bigoplus_{k=0}^n\mathcal{A}^{(k,k+1)}(x,\dot{x})\in\mathfrak{g}.$

\vspace{.3cm}
\paragraph{General case.} \label{casogeneral} Consider again the case in which $\mathcal{S}$ cannot be writen as
$\displaystyle\mathcal{S}:=\mathcal{S}_e\equiv\mathcal{S}_{(0,1)}\oplus\mathcal{S}_{(1,2)}\oplus...\oplus\mathcal{S}_{(n,n+1)}$
 being \vspace{.2cm} $\mathcal{S}_{(i,i+1)}=\mathcal{S}\cap\mathfrak{n}_{(i,i+1)},$ for
$0\leq i\leq n.$ Then, using the same notation in page
\pageref{elemAlgLie}, if $\xi$ is a generic element of
$\mathcal{S},$ it has the following form
$\xi=(\varepsilon^0_1,...,\varepsilon^0_{s_0},h^0_{s_0+1},...,h^0_{d_0-s_0},\varepsilon^1_1,...,\varepsilon^1_{s_1},h^1_{s_1+1},...,h^1_{d_1-s_1}
,...,\varepsilon^n_1,...,\varepsilon^n_{s_n},h^n_{s_n+1},...,h^n_{d_n-s_n}).$

\vspace{.3cm} Then, if $\displaystyle\xi\in\mathcal{S},$ the
\textbf{\emph{Lagrange-d'Alembert-Poincar\'e equations by $n$ stages}}
in this case are:
$$\begin{array}{c}
\hspace{-1.5cm}\displaystyle\sum_{j=0}^n\dot{\mu}_{j}(\varepsilon^j_1,...,\varepsilon^j_{s_j},h^j_{s_j+1},...,h^j_{d_j-s_j})=
\displaystyle\sum_{i=0}^n\left\langle\mu_i,
[\eta^{(i,i+1)},
(\varepsilon^i_1,...,\varepsilon^i_{s_i},h^i_{s_i+1},...,h^i_{d_i-s_i})] \right.\\
\displaystyle\left. -[\mathcal{A}^{(i,i+1)}(x,\dot{x}),
(\varepsilon^i_1,...,\varepsilon^i_{s_i},h^i_{s_i+1},...,h^i_{d_i-s_i})]+ \sum_{j=0}^{i-1}\left(-a_{(N_j,N_{j+1})}^{(i,i+1)}
(\eta^{(j,j+1)},(\varepsilon^j_1,...,\varepsilon^j_{s_j},h^j_{s_j+1},...,h^j_{d_j-s_j}))+\right. \right. \\
\displaystyle\left.\left. b_{(N_j,N_{j+1})}^{(i,i+1)}\left(\eta^{(j,j+1)},\sum_{k=j+1}^{n}
(\varepsilon^k_1,...,\varepsilon^k_{s_k},h^k_{s_k+1},...,h^k_{d_k-s_k})\right) \right.\right.\\
\displaystyle\left.\left.  - b_{(N_j,N_{j+1})}^{(i,i+1)}\left((\varepsilon^j_1,...,\varepsilon^j_{s_j},h^j_{s_j+1},...,h^j_{d_j-s_j}),\sum_{l=j+1}^{n}
\eta^{(l,l+1)} \right)+ \right.\right.\\
\displaystyle \left.\left. a_{(N_j,N_{j+1})}^{(i,i+1)}(\mathcal{A}^{(j,j+1)}(x,\dot{x}),(\varepsilon^j_1,...,\varepsilon^j_{s_j},h^j_{s_j+1},...,h^j_{d_j-s_j}))\right. \right.
\nonumber
\end{array}$$
\begin{equation}\label{L-d'A-Pstagesgral}
\displaystyle\left.\left.-
 b_{(N_j,N_{j+1})}^{(i,i+1)}\left(\mathcal{A}^{(j,j+1)}(x,\dot{x}),\sum_{k=j+1}^{n}
(\varepsilon^k_1,...,\varepsilon^k_{s_k},h^k_{s_k+1},...,h^k_{d_k-s_k})\right)+ \right.\right.
\end{equation}
$$\begin{array}{c}
\displaystyle\left.\left. b_{(N_j,N_{j+1})}^{(i,i+1)}\left((\varepsilon^j_1,...,\varepsilon^j_{s_j},h^j_{s_j+1},...,h^j_{d_j-s_j}),\sum_{l=j+1}^{n}
\mathcal{A}^{(l,l+1)}(x,\dot{x})\right)\right)\right\rangle,\hspace{2cm}\\
\displaystyle\frac{\partial l}{\partial x}\ \delta x-\frac{d}{dt}\frac{\partial l}{\partial\dot{x}}\ \delta x=
\displaystyle\sum_{i=0}^n\left\langle\mu_i,
\varphi^{(i,i+1)}-[\mathcal{A}^{(i,i+1)}(x,\delta x),\eta^{(i,i+1)}]-\right.\hspace{4cm}\\
\displaystyle\left.
\sum_{j=0}^{i-1}\left(-a_{(N_j,N_{j+1})}^{(i,i+1)}(\mathcal{A}^{(j,j+1)}(x,\delta x),\eta^{(j,j+1)})+
 b_{(N_j,N_{j+1})}^{(i,i+1)}\left(\mathcal{A}^{(j,j+1)}(x,\delta x),\sum_{k=j+1}^{n}
\eta^{(k,k+1)}\right)\right.\right.\\
\displaystyle\left.\left.
-b_{(N_j,N_{j+1})}^{(i,i+1)}\left(\eta^{(j,j+1)},\sum_{l=j+1}^{n}
\mathcal{A}^{(l,l+1)}(x,\delta x)\right) \right)\right\rangle,\\
\displaystyle\mu_i=\frac{\partial
l}{\partial v_i}(v),\quad v_j\in\mathcal{S}_j,\hspace{11.3cm}\\
\displaystyle v=\bigoplus_{i=0}^n\eta^{(i,i+1)},\hspace{13.1cm}\\
\displaystyle\mathcal{A}(x)\dot{x}=\bigoplus_{k=0}^n\mathcal{A}^{(k,k+1)}(x,\dot{x}),\hspace{11.5cm}\\
\displaystyle\hspace{.1cm} B(x,e)((\dot{x},0),(\delta
x,0))=\bigoplus_{k=0}^n\varphi^{(k,k+1)},\hspace{10cm}
\end{array}$$

\vspace{.3cm}
for all $\delta x\in T_x(Q/G)$ and $\varepsilon^i_{r_i}\in\mathcal{S}_{(i,i+1)},$ for $i=0,...,n.$

\vspace{.5cm} We proceed as before in the general case in section \ref{EDPBSS} to obtain the equations of motion.

\section{Example of reduction by stages}\label{sect:disco}
\subsection{The Lagrange-d'Alembert-Poincar\'e equations by stages for Euler's disk}

Let us consider an Euler's disk of radius $r$ and thickness $re$
rolling on a rough horizontal plane and having only one point of
contact with it, in other words, the horizontal and vertical positions of the disk is excluded from the configuration space. We will use the description and notation of reference \cite{CDgrueso}, however here we shall apply different methods.

The configuration space for Euler's disk is
\begin{eqnarray}\label{Q}
Q=\left(0,\frac{\pi}{2}\right)\times S^1\times
S^1\times\mathbb{R}^2,
\end{eqnarray}
and a point $q\in Q$ is written as $q=(\theta,\varphi,\psi,x)$. Let denote $A$ an orthogonal $3\times 3$ matrix representing the instantaneous orientation of the disk, therefore $A \textbf{e}_3$  has $\textbf{e}_3$ component belonging to $(0,1)$.
The angle $\theta$ is the angle from the axis
$\textbf{e}_3$ to the vector $z=A\textbf{e}_3$. The vector $y$ is the unit vector directed from
the point of contact with the plane to the origin of the system
$(A\textbf{e}_1,A\textbf{e}_2,A\textbf{e}_3).$ The vector $u$ is
defined by $u=z\times y$ being tangent to the disk at the point of
contact $x$ and having the direction of the motion of this point on
the plane. The unit vector $u$ has the expression
$u=(-\cos\varphi,-\sin\varphi,0)$ which defines the angle $\varphi.$
The angle $\psi$ is the angle from the vector $-y$ to the vector
$A\textbf{e}_1$ where the positive sense for measuring the angle
$\psi$ on the plane of the disk is the counterclockwise sense, as
viewed from $z.$

\vspace{.2cm}The nonholonomic constraint is given by the
distribution
\begin{eqnarray}\label{distribution}
\mathcal{D}_{(\theta,\varphi,\psi,x)}=\left\{(\theta,\varphi,\psi,x,\dot{\theta},\dot{\varphi},\dot{\psi},
\dot{x})\mid \dot{x}=\dot{\psi}ru\right\}.
\end{eqnarray}

\vspace{.2cm} The symmetry group that we are going to consider is
$G=SO(2)\times\mathbb{R}^2\equiv
S^1\times\mathbb{R}^2.$ Since the group is abelian, we can consider
that the group $G$ acts on the left
on the configuration space $Q$ by the action
$$\displaystyle (\alpha,
a)(\theta,\varphi,\psi,x)=(\theta,\varphi,\psi+\alpha,x+a).$$ With
this action $Q$ is a left principal bundle with structure group
$G=S^1\times\mathbb{R}.$ The map $\pi:Q\rightarrow
Q/G\equiv(0,\pi/2)\times S^1,$ given by
$\pi(\theta,\varphi,\psi,x)=(\theta,\varphi),$ is a submersion and we have
$[(\theta,\varphi,\psi,x)]_{S^1\times\mathbb{R}^2}\equiv(\theta,\varphi).$

\vspace{.2cm} The vertical distribution $\mathcal{V}$ is given by
\begin{eqnarray}\label{vertdistrib}
\mathcal{V}_{(\theta,\varphi,\psi,x)}=\{(\theta,\varphi,\psi,x,\dot{\theta},\dot{\varphi},\dot{\psi},\dot{x})\mid
\dot{\theta}=0,\dot{\varphi}=0\},
\end{eqnarray}

and the vector bundle $\mathcal{S}=\mathcal{D}\cap\mathcal{V}$ is
\begin{eqnarray}\label{ese}
\mathcal{S}_{(\theta,\varphi,\psi,x)}=\{(\theta,\varphi,\psi,x,0,0,\xi,\xi
ru)\}.
\end{eqnarray}

Since $\dim\mathcal{D}_{(\theta,\varphi,\psi,x)}=3,\
\dim\mathcal{V}_{(\theta,\varphi,\psi,x)}=3$ and
$\dim\mathcal{S}_{(\theta,\varphi,\psi,x)}=1,$ we have \newline
$\displaystyle\mathcal{D}_{(\theta,\varphi,\psi,x)}+\mathcal{V}_{(\theta,\varphi,\psi,x)}=T_{(\theta,\varphi,\psi,x)}Q,$
\hspace{.1cm} that is, the dimension assumption is satisfied.

\vspace{.2cm} We choose the horizontal spaces
$$\displaystyle\mathcal{H}_{(\theta,\varphi,\psi,x)}=\{(\theta,\varphi,\psi,x,\alpha,\beta,0,0)\}$$
satisfying
$\displaystyle\mathcal{H}_{(\theta,\varphi,\psi,x)}\oplus\mathcal{S}_{(\theta,\varphi,\psi,x)}=
\mathcal{D}_{(\theta,\varphi,\psi,x)}.$ So the distribution
$\mathcal{H}$ is $G$-invariant.

\vspace{.3cm} The connection 1-form $\mathcal{A}$ whose horizontal
spaces are $\mathcal{H}_{(\theta,\varphi,\psi,x)}$ is
\begin{eqnarray}\label{conexionA}
\mathcal{A}(\theta,\varphi,\psi,x,\dot{\theta},\dot{\varphi},\dot{\psi},\dot{x})=(\dot{\psi},
\dot{x}).
\end{eqnarray}

\vspace{.5cm} In this context, we are going to apply the techniques
of reduction by stages to the example of Euler's disk. We shall
write the Lagrange-d'Alembert-Poincar\'e equations by stages for this
example.

We choose $N_1=\mathbb{R}^2$ as a normal subgroup of the symmetry
group $G=S^1\times\mathbb{R}^2,$ and $N_2=\{e\}$ as a normal
subgroup of $N_1.$

So, following the notation of the theorem \ref{formcorchete} in page
\pageref{formcorchete}, we have $N_{(0,1)}=N_0/N_1=G/N_1=S^1$ and
$N_{(1,2)}=N_1/N_2=\mathbb{R}^2.$ Then we can write the chain of
normal subgroups as follows
\begin{eqnarray}\label{chainG}
N_2=\{e\}\lhd N_1=\mathbb{R}^2\lhd N_0=G=S^1\times\mathbb{R}^2.
\end{eqnarray}

\vspace{.2cm} We note $\mathfrak{g},\ \mathfrak{n}_1$ and
$\mathfrak{n}_2$ the Lie algebras of the groups $G,\ N_1$ and $N_2$
respectively. And $\mathfrak{n}_{(0,1)},\ \mathfrak{n}_{(1,2)}$ denote the Lie
algebras of the groups $N_{(0,1)}$ and $N_{(1,2)}.$

We consider the identification of the Lie algebra of the symmetry
group $G,$ as in (\ref{descompg3}),
\begin{eqnarray}\label{decompalgG}
\mathfrak{g}\equiv\mathfrak{n}_{(0,1)}\oplus\mathfrak{n}_{(1,2)}=\mathfrak{s}^1\oplus\mathbb{R}^2.
\end{eqnarray}

So $m=\dim(\mathfrak{g})=3,\ d_0=\dim(\mathfrak{n}_{(0,1)})=1$ and
$d_1=\dim(\mathfrak{n}_{(1,2)})=2.$

\vspace{.2cm} In this case,
$\mathcal{S}_{(\theta,\varphi,\psi,x)}=\{(\theta,\varphi,\psi,x,0,0,\xi,\xi
ru)\}$ cannot be decomposed in
$\mathcal{S}_{(0,1)}\oplus\mathcal{S}_{(1,2)},$ being
$\mathcal{S}_{(0,1)}=\mathcal{S}\cap\mathfrak{n}_{(0,1)}$ and
$\mathcal{S}_{(1,2)}=\mathcal{S}\cap\mathfrak{n}_{(1,2)}.$ Then we
have the general case considered in section \ref{sect:Lag-d'Alem-Poin}.

\vspace{.2cm} Then a generic element
$\mu\in\mathcal{S}_{(\theta,\varphi,\psi,x)}$ has the form
$\mu=(\varepsilon^0_1,h^1_1(x_1),h^1_2(x_1)),$ where $\varepsilon^0_1=\xi,$
$h^1_1(x_1)=-\xi\ r\ cos\varphi$ and $h^1_2(x_1)=-\xi\ r\
sin\varphi.$

And the Lagrange-d'Alembert-Poincar\'e equations by stages (see
formula (\ref{L-d'A-Pstagesgral})) for Euler's disk are:

$$\dot{\mu_0}(\xi)+\dot{\mu}_1(-\xi\ r\ cos\varphi,-\xi\ r\ sin\varphi)=\hspace{7cm}$$ $$=\left\langle\mu_0,\left[\eta^{(0,1)}-\mathcal{A}^{(0,1)}(\theta,\varphi,\dot{\theta},\dot{\varphi}),\xi\right]\right\rangle+\left\langle\mu_1,\left[
\eta^{(1,2)}-\mathcal{A}^{(1,2)}(\theta,\varphi,\dot{\theta},\dot{\varphi}),(-\xi\
r\ cos\varphi,-\xi\ r\ sin\varphi)\right]+\right.$$
$$\left.b^{(1,2)}_{(N_0,N_1)}\left(\eta^{(0,1)}-\mathcal{A}^{(0,1)}(\theta,\varphi,\dot{\theta},\dot{\varphi}),(-\xi\
r\ cos\varphi,-\xi\ r\
sin\varphi)\right)-b^{(1,2)}_{(N_0,N_1)}\left(\xi,\eta^{(1,2)}-\mathcal{A}^{(1,2)}(\theta,\varphi,\dot{\theta},\dot{\varphi})\right)\right.$$
$$\left.-a^{(1,2)}_{(N_0,N_1)}\left(\eta^{(0,1)}-
\mathcal{A}^{(0,1)}(\theta,\varphi,\dot{\theta},\dot{\varphi}),\xi\right)\right\rangle\
,\hspace{4cm}$$

$$\frac{\partial l}{\partial\theta}\ \delta\theta-\frac{d}{dt}\frac{\partial l}{\partial\dot{\theta}}\ \delta\theta=
\left\langle\mu_0,B^{(0,1)}(\theta,0,0,0)\left((\dot{\theta},0,0,0),(\delta\theta,0,0,0)\right)-
\left[\mathcal{A}^{(0,1)}(\theta,0,\delta\theta,0),\eta^{(0,1)}\right]\right\rangle+$$
$$\left\langle\mu_1,B^{(1,2)}(\theta,0,0,0),\left((\dot{\theta},0,0,0),(\delta\theta,0,0,0)\right)-
\left[\mathcal{A}^{(1,2)}(\theta,0,\delta\theta,0),
\eta^{(1,2)}\right]\right.$$
$$\left.-b^{(1,2)}_{(N_0,N_1)}\left(\mathcal{A}^{(0,1)}(\theta,0,\delta\theta,0),\eta^{(1,2)}\right)-b^{(1,2)}_{(N_0,N_1)}\left(\eta^{(0,1)},
\mathcal{A}^{(1,2)}
(\theta,0,\delta\theta,0)\right)\right\rangle\ ,$$

$$\frac{\partial l}{\partial\varphi}\ \delta\varphi-\frac{d}{dt}\frac{\partial l}{\partial\dot{\varphi}}\
\delta\varphi=\left\langle\mu_0,B^{(0,1)}(0,\varphi,0,0)\left((0,\dot{\varphi},0,0),(0,\delta\varphi,0,0)\right)-
\left[\mathcal{A}^{(0,1)}(0,\varphi,0,\delta\varphi),\eta^{(0,1)}\right]\right\rangle+$$
$$\left\langle\mu_1,B^{(1,2)}(0,\varphi,0,0),\left((0,\dot{\varphi},0,0),(0,\delta\varphi,0,0)\right)-
\left[\mathcal{A}^{(1,2)}(0,\varphi,0,\delta\varphi),\eta^{(1,2)}\right]\right.$$

$$\left.-b^{(1,2)}_{(N_0,N_1)}\left(\mathcal{A}^{(0,1)}(0,\varphi,0,\delta\varphi),\eta^{(1,2)}\right)-b^{(1,2)}_{(N_0,N_1)}
\left(\eta^{(0,1)},\mathcal{A}^{(1,2)}(0,\varphi,0,\delta\varphi)\right)\right\rangle\
,$$

where
\begin{eqnarray*}
\mu_0 &=& \frac{\partial l}{\partial\eta^{(0,1)}}(\eta^{(0,1)}\oplus\eta^{(1,2)}),\quad \text{with} \quad
\eta^{(0,1)}\oplus\eta^{(1,2)}\in\mathcal{S}_{(\theta,\varphi,\psi,x)},\\
\mu_1 &=& \frac{\partial l}{\eta^{(1,2)}}(\eta^{(0,1)}\oplus\eta^{(1,2)}),\\
\mathcal{A}(\theta,\varphi,\psi,x)&=& \mathcal{A}^{(0,1)}(\theta,\varphi,\psi,x)\oplus\mathcal{A}^{(1,2)}
(\theta,\varphi,\psi,x),\\
B(\theta,\varphi,\psi,x)\left((\dot{\theta},\dot{\varphi},0,0),(\delta\theta,\delta\varphi,0,0)
\right) &= &
B^{(0,1)}(\theta,\varphi,\psi,x)\left((\dot{\theta},\dot{\varphi},0,0),(\delta\theta,\delta\varphi,0,0)\right)\oplus\\
& & B^{(1,2)}(\theta,\varphi,\psi,x)\left((\dot{\theta},\dot{\varphi},0,0),(\delta\theta,\delta\varphi,0,0)\right).
\end{eqnarray*}

\vspace{.5cm} Since $\mathcal{A}(\dot{\theta},\dot{\varphi})=(0,0)$
(see equation (\ref{conexionA})), we have \newline
$\displaystyle\mathcal{A}^{(0,1)}(\theta,\varphi,\dot{\theta},\dot{\varphi})=\mathcal{A}^{(1,2)}(\theta,\varphi,\dot{\theta},
\dot{\varphi})=0$ and
$B(\theta,\varphi,0,0)\left((\dot{\theta},\dot{\varphi},0,0),(\delta\theta,\delta\varphi,0,0)\right)=$ \newline $d\mathcal{A}
(\theta,\varphi)(\dot{\theta},\dot{\varphi},\delta\theta,\delta\varphi)-\left[\mathcal{A}(\theta,\varphi)
(\dot{\theta},\dot{\varphi}),\mathcal{A}(\theta,\varphi)(\delta\theta,\delta\varphi)\right]=0.$

\vspace{.3cm}
Then, the equations that we obtain are
\begin{eqnarray*}
\dot{\mu_0}(\xi)+\dot{\mu}_1(\xi ru)&=& \left\langle\mu_0,[\eta^{(0,1)},\xi]\right\rangle+\left\langle
\mu_1,[\eta^{(1,2)},\xi ru]+b^{(1,2)}_{(N_0,N_1)}(\eta^{(0,1)},\xi ru) \right.\\
& & \left. -b^{(1,2)}_{(N_0,N_1)}(\xi,\eta^{(1,2)})-a^{(1,2)}_{(N_0,N_1)}(\eta^{(0,1)},\xi)\right\rangle,\\
\frac{\partial l}{\partial \theta}\ \delta\theta-\frac{d}{dt}\frac{\partial l}{\partial\dot{\theta}}\ \delta\theta &=&
\langle\mu_0,0\rangle+\left\langle\mu_1,-b^{(1,2)}_{(N_0,N_1)}(0,\eta^{(1,2)})-b^{(1,2)}_{(N_0,N_1)}(\eta^{(0,1)},0)
\right\rangle,\\
\frac{\partial l}{\partial \varphi}\ \delta\varphi-\frac{d}{dt}\frac{\partial l}{\partial\dot{\varphi}}\ \delta\varphi &=& \langle\mu_0,0\rangle+\left\langle\mu_1,-b^{(1,2)}_{(N_0,N_1)}(0,\eta^{(1,2)})-b^{(1,2)}_{(N_0,N_1)}(\eta^{(0,1)},0)
\right\rangle,\\
\mu_0 &= &\frac{\partial l}{\partial\eta^{(0,1)}}(\eta^{(0,1)},\eta^{(1,2)}),\\
\mu_1 &= &\frac{\partial l}{\partial\eta^{(1,2)}}(\eta^{(0,1)},\eta^{(1,2)}),\quad\text{with}\quad
\eta^{(1,2)}=\eta^{(0,1)} ru.
\end{eqnarray*}

\vspace{.5cm} On the other hand, using the fact that the groups
$N_{(0,1)}=S^1$ and $N_{(1,2)}=\mathbb{R}^2$ are abelian, we can write the
equations as follows

\begin{eqnarray*}
\dot{\mu_0}(\xi)+\dot{\mu}_1(\xi ru)&=& \left\langle\mu_1,b^{(1,2)}_{(N_0,N_1)}(\eta^{(0,1)},\xi ru)-
b^{(1,2)}_{(N_0,N_1)}(\xi,\eta^{(1,2)})-a^{(1,2)}_{(N_0,N_1)}(\eta^{(0,1)},\xi)\right\rangle,\\
\frac{\partial l}{\partial\theta}\ \delta\theta-\frac{d}{dt}\frac{\partial l}{\partial\dot{\theta}}\ \delta\theta &=&
-\left\langle\mu_1,b^{(1,2)}_{(N_0,N_1)}(0,\eta^{(1,2)})+b^{(1,2)}_{(N_0,N_1)}(\eta^{(0,1)},0)\right\rangle,\\
\frac{\partial l}{\partial\varphi}\ \delta\varphi-\frac{d}{dt}\frac{\partial l}{\partial\dot{\varphi}}\
\delta\varphi &=&
-\left\langle\mu_1,b^{(1,2)}_{(N_0,N_1)}(0,\eta^{(1,2)})+b^{(1,2)}_{(N_0,N_1)}(\eta^{(0,1)},0)\right\rangle,\\
\mu_0 &=& \frac{\partial l}{\partial\eta^{(0,1)}}(\eta^{(0,1)},\eta^{(1,2)}),\\
\mu_1 &=& \frac{\partial l}{\partial(\eta^{(1,2})}(\eta^{(0,1)},\eta^{(1,2)}).
\end{eqnarray*}

\vspace{.5cm} According to the definitions of the bilinear forms $b$
and $a$ given in theorem \ref{formcorchete}, we have that
$\quad \displaystyle
b_{(N_0,N_1)}(\kappa,\eta)=\left[\left[e,\left[\kappa^{\mathcal{A}_{N_1}},\eta\right]\right]_{N_1}\right]_{N_{(0,1)}}$
\hspace{.2cm} and \newline $\displaystyle
a_{(N_0,N_1)}(\kappa,\xi)=\left[\left[e,-\mathcal{A}_{N_1}(e)\left(\left[\kappa^{\mathcal{A}_{N_1}},\xi^{\mathcal{A}_{N_1}}
\right]\right)\right]_{N_1}\right]_{N_{(0,1)}}.$

\vspace{.3cm} In this case, the corresponding $G$-invariant 1-form
connection
$\mathcal{A}_{\mathfrak{n}_1}:T_gG\rightarrow\mathfrak{n}_1$  over the
principal bundle $Q$ with structure group $N_1$ (see the first paragraph of the section \ref{sec:LieBracket}), is given by
$\displaystyle\mathcal{A}_{\mathfrak{n}_1}(\psi,x,\dot{\psi},\dot{x})=\dot{x}.$

\vspace{.2cm} So, we obtain \hspace{.5cm} $\displaystyle
b_{(N_0,N_1)}(\kappa,\eta)=\left[\left[e,\left[(\kappa,0),(0,\eta)\right]\right]_{N_1}\right]_{N_{(0,1)}}=0$
\hspace{.5cm} and

\vspace{.3cm} $\displaystyle
a_{(N_0,N_1)}(\kappa,\xi)=\left[\left[e,-\mathcal{A}_{N_1}(e)\left(\left[(\kappa,0),(\xi,0)\right]\right)\right]_{N_1}
\right]_{N_{(0,1)}}=0.$

\vspace{.5cm} Then, the expression of the
Lagrange-d'Alembert-Poincar\'e equations by stages for Euler's
disk is

\begin{eqnarray}\label{ecetapasdisco}
\frac{d}{dt}\frac{\partial l}{\partial\eta^{(0,1)}}(\xi)+\frac{d}{dt}
\frac{\partial l}{\partial\eta^{(1,2)}}(\xi ru)&=&0, \nonumber\\
\frac{\partial l}{\partial\theta}\
\delta\theta-\frac{d}{dt}\frac{\partial l}{\partial\dot{\theta}}\
\delta\theta &=& 0,\\
\frac{\partial l}{\partial\varphi}\
\delta\varphi-\frac{d}{dt}\frac{\partial l}{\partial\dot{\varphi}}\
\delta\varphi &=& 0.\nonumber
\end{eqnarray}

for all $\xi,\ \delta\theta, \ \delta\varphi\in\mathbb{R},$

\vspace{.5cm} The reduced Lagrangian for Euler's disk is given by

\vspace{.5cm} $\displaystyle
\ell(\theta,\varphi,\dot{\theta},\dot{\varphi},\eta^{(0,1)},\eta^{(1,2)})
=-M g r \sin\theta+ \frac{1}{2}\left(I_1+\frac{1}{4}M r^2
e^2\right)(\dot{\theta}^2+\dot{\varphi}^2\sin^2\theta)+ \frac{1}{2}
M (\eta^{(1,2)})^2 +$

$\displaystyle\frac{1}{2} M r^2\dot{\theta}^2+\frac{1}{2}M
r^2\dot{\varphi}^2\cos^2\theta+M r\
\eta^{(1,2)}_1(\dot{\theta}\sin\theta\sin\varphi-\dot{\varphi}\cos\theta\cos\varphi)-M
r\ \eta^{(1,2)}_2(\dot{\theta}\cos\varphi\sin\theta+$

$\displaystyle\dot{\varphi}\cos\theta\sin\varphi)+\frac{1}{2}I_3(\dot{\varphi}
\cos\theta+\eta^{(0,1)})^2+\frac{1}{2}M r e
\left[\eta^{(1,2)}_1(\dot{\varphi}\cos\varphi\sin\theta+\dot{\theta}\sin\varphi\cos\theta)+\right.$

$\displaystyle\left.\eta^{(1,2)}_2(\dot{\varphi} \sin\varphi
\sin\theta-\dot{\theta}\cos\theta\cos\varphi)\right.$ $\displaystyle
\left.-r \cos\theta\sin\theta \dot{\varphi}^2\right]-\frac{1}{2}M g
r e \cos{\theta},$

\vspace{.5cm} with $\eta^{(1,2)}=\left(\eta^{(1,2)}_1,\
\eta^{(1,2)}_2\right)\in\mathbb{R}^2.$

\vspace{.3cm} So we can
write the Lagrange-d'Alembert-Poincar\'e equations by stages  (equations (\ref{ecetapasdisco})) for
Euler's disk as follows

\begin{eqnarray*}{rl}
I_3(\dot{\eta}^{(0,1)}+\ddot{\varphi}\cos{\theta}-\dot{\theta}\dot{\varphi}\sin\theta)+M
r \langle\dot{\eta}^{(1,2)},u\rangle+M r^2
(\ddot{\varphi}\cos\theta-2\dot{\theta}\dot{\varphi}\sin\theta)& &\\
-\frac{1}{2}M
e r^2
(\ddot{\varphi}\sin\theta+2\dot{\theta}\dot{\varphi}\cos\theta)&=&0,\\
2(M
r^2+I_3-I_1)\dot{\theta}\dot{\varphi}\sin\theta\cos\theta-(I_1\sin^2\theta+(I_3+M
r^2)\cos^2\theta)\ddot{\varphi}+M r
\cos\theta & &\nonumber \\
(\dot{\eta}^{(1,2)}_1\cos{\varphi}+
\dot{\eta}^{(1,2)}_2\sin{\varphi})+ I_3
(\dot{\theta}\eta^{(0,1)}\sin\theta-\dot{\eta}^{(0,1)}\cos{\theta})& &\\
-\frac{1}{2}Mr e \left(\frac{1}{2}r e
(2\dot{\theta}\dot{\varphi}\sin\theta\cos\theta+\ddot{\varphi}\sin^2\theta)+
\sin\theta(\dot{\eta}^{(1,2)}_1\cos\varphi+\dot{\eta}^{(1,2)}_2\sin\varphi)+\right.& &\\
\left. 2r\dot{\theta}\dot{\varphi}
(\sin^2\theta-\cos^2\theta)-2r\ddot{\varphi}\sin\theta\cos\theta\right)&=& 0,\\
(I_1-I_3-M
r^2)\dot{\varphi}^2\sin\theta\cos\theta-I_3\eta^{(0,1)}\dot{\varphi}\sin\theta-(I_1+M
r^2)\ddot{\theta} +& &\\
M r \sin{\theta}(\dot{\eta}^{(1,2)}_2\cos{\varphi}-\dot{\eta}^{(1,2)}_1\sin{\varphi})-M g r \cos{\theta}+
\nonumber \\\frac{1}{2}M r e \left(\frac{1}{2}r
e \dot{\varphi}^2
\sin\theta\cos\theta+r\dot{\varphi}^2(\sin^2\theta-\cos^2\theta)+
\right.& & \nonumber \\
\left.\cos\theta(\dot{\eta}^{(1,2)}_2\cos\varphi-\dot{\eta}^{(1,2)}_1\sin\varphi)+g
\sin\theta-\frac{1}{2}r e \ddot{\theta}\right)&=&0,\nonumber\\
\eta^{(0,1)} r &=&\eta^{(1,2)},
\end{eqnarray*}

where the last equation is equivalent to the condition
$\displaystyle(\eta^{(0,1)},\eta^{(1,2)})\in\mathcal{S}_{(\theta,\varphi,\psi,x)}.$

\vspace{.5cm}
This system of
equations is the same that we have obtained in \cite{CDgrueso}, where we developed reduction in one stage.


\section{Appendix}
\renewcommand{\thesubsection}{\Alph{subsection}.}

\subsection{Reduced connections}\label{appendixa}
We shall recall some known facts on associated bundles and covariant derivatives. Proofs can be found in
\cite{CMR01a}.

Let a left principal bundle $\pi:Q\rightarrow Q/G$ and $A$ a principal connection. Consider a left action $\rho:G\times M\rightarrow M$ of the Lie group $G$ on a manifold $M.$ The \emph{associated bundle with standard fiber $M$} \index{Fibrado!asociado|textbf} is, by definition, $Q\times_G M=(Q\times M)/G,$ where the action of $G$ on $Q\times M$ is given by $g(q,m)=(gq,gm).$ The class (or orbit) of $(q,m)$ is denoted $[q,m]_G$. The \emph{projection} \index{Proyecci\'on!en fibrado
asociado|textbf} $\pi_M:Q\times_G M\rightarrow Q/G$ is defined by $\pi_M([q,m]_G)=\pi(q)$ and it is easy to check that is well defined and is a surjective submersion.

\vspace{.3cm} Let $[q_0,m_0]_G\in Q\times_G M$ and let $x_0=\pi(q_0)\in Q/G.$ Let $x(t),$ $t\in[a,b],$ be a curve on $Q/G$ and let $t_0\in[a,b]$ be such that $x(t_0)=x_0.$ The \emph{parallel transport} \index{Transporte paralelo!|textbf} of this element $[q_0,m_0]_G$ along  the curve $x(t)$ is defined to be the curve $[q,m]_G(t):=[x^h_{q_0}(t),m_0]_G,$
where $x^h_{q_0}(t)$ is the horizontal lift of the curve $x(t)$ with initial condition $q_0$.

For $t,t+s\in[a,b],$ we adopt the notation $\tau^t_{t+s}:\pi^{-1}_M(x(t))\rightarrow\pi^{-1}_M(x(t+s))$ for the parallel transport map along the curve $x(s)$ of any point $[q(t),m(t)]_G\in\pi^{-1}_M(x(t))$ to the corresponding point $\tau^t_{t+s}[q(t),m(t)]_G\in\pi^{-1}_M(x(t+s)).$
 Thus,
$\tau^t_{t+s}[q(t),m(t)]_G=$

$[x^h_{q(t)}(t+s),m(t)]_G.$

From now on, we assume that $M$ is a vector space and $\rho$ is a linear representation. In this case, the associated bundle with standard fiber $M$ is naturally a vector bundle. We will use the identification $TM=M\times M.$ We shall sometimes use the notation $\rho'(\xi)$ for the second component of the infinitesimal generator of an element $\xi\in\mathcal{G},$ that is, $\xi m=(m,\rho'(\xi)(m)).$  Thus, we have a linear representation of the Lie algebra $\mathcal{G}$ on the vector space $M$, $\rho':\mathcal{G}\rightarrow \text{End}(M)$ (the linear vector fields on $M$ are identified with the space of linear maps of $M$ to itself).

\begin{dfn}\label{derivcovar}
    Let $[q(t),m(t)]_G,\ t\in[a,b],$ be a curve on $Q\times_G M,$ denote by
\newline $\displaystyle x(t)=\pi_M([q(t),m(t)]_G)=\pi(q(t))$ its projection on the base $Q/G,$ and let $\tau^t_{t+s},$ denote the parallel transport along $x(t)$ from time $t$ to time $t+s,$ where $t,t+s\in[a,b]$.
 The \textbf{covariant derivative} of $[q(t),m(t)]_G$ along $x(t)$ is defined as follows $$\frac{D[q(t),m(t)]_G}{Dt}=\lim_{s\rightarrow 0}\frac{\tau^{t+s}_t([q(t+s),m(t+s)]_G)-[q(t),m(t)]_G}{s}.$$
\end{dfn}

Thus, the covariant derivative of $[q(t),m(t)]_G$ is an element of $\pi^{-1}_M(x(t)).$

Note that if $[q(t),m(t)]_G$ is a vertical curve, then the covariant derivative in the asso\-cia\-ted bundle is just the fiber derivative, since its base point is constant. That is,  $$\frac{D[q(t),m(t)]_G}{Dt}=[q(t),m'(t)]_G,$$ where $m'(t)$ is the time derivative of $m.$

By definition, a \emph{connection} (sometimes called an \emph{affine connection}) \index{Conexi\'on!af\'\in|textbf} $\nabla$ on a vector bundle $\tau:V\rightarrow Q$ is a
map $\nabla:\mathfrak{X}^{\infty}(Q)\times
\Gamma(V)\rightarrow \Gamma(V),$ say $(X,v)\mapsto \nabla_X v,$ having the following properties:
\begin{enumerate}
  \item $\displaystyle\nabla_{f_1X_1+f_2X_2}v=f_1\nabla_{X_1}v+f_2\nabla_{X_2}v$
   for all $X_i\in\mathfrak{X}^{\infty}(Q)$ (the space of smooth vector fields on $Q$), $f_i\in C^{\infty}(Q)$ (the space of smooth real valued functions on $Q$), $i=1,2,$ and all $v\in\Gamma(V)$ (the space of smooth sections of the vector bundle $V$);
  \item $\displaystyle\nabla_X(f_1v_1+f_2v_2)=X[f_1]v_1+f_1\nabla_X v_1+X[f_2]v_2+f_2\nabla_X
  v_2$ for all $X\in\mathfrak{X}^{\infty}(Q),$ \newline $f_i\in
  C^{\infty}(Q),$ and $v_i\in\Gamma(V),\ i=1,2;$ where $X[f]$ denotes the derivative of $f$ in the direction of the vector field $X.$
\end{enumerate}

Given a connection on $V,$ the \emph{parallel transport}
\index{Transporte paralelo!a lo largo de una curva|textbf} of a vector $v_0\in \tau^{-1}(q_0)$ along a curve $q(t)$ in $Q,\ t\in[a,b],$ such that $q(t_0)=q_0$ for a fixed $t_0\in[a,b],$ is the unique curve $v(t)$ such that $v(t)\in\tau^{-1}(q(t))$ for all $t,$ $v(t_0)=v_0,$ and which satisfies $\displaystyle\nabla_{\dot{q}(t)}v(t)=0$ for all $t.$ The operation of parallel transport establishes for each $t,s\in[a,b],$ a linear map $T^t_{t+s}:\tau^{-1}(q(t))\rightarrow \tau^{-1}(q(t+s))$ associated to each curve $q(t)$ in $Q.$ Then we can define the operation of \emph{covariant derivative}
\index{Derivada!covariante!de una curva|textbf} on curves $v(t)$ in $V$ similar to that in the previous definition \ref{derivcovar}; that is,
$$\left.\frac{Dv(t)}{Dt}=\frac{d}{ds}T^{t+s}_t v(t+s)\right|_{
s=0}.$$

Observe that the connection $\nabla$ can be recovered from the covariant derivative (and thus from the parallel transport operation). Indeed, $\nabla$ is given by $$\nabla_Xv(q_0)=\left.\frac{D}{Dt}v(t)\right|_{t=t_0},$$ where for each
$q_0\in Q,$ each $X\in\mathfrak{X}^{\infty}(Q)$ and each $v\in\Gamma(V),$ we have, by definiton, that $q(t)$ is any curve in $Q$ such that $\dot{q}(t_0)=X(q_0)$ and $v(t)=v(q(t))$ for all $t.$ This property es\-ta\-bli\-shes, in particular, the uniqueness of the connection associated to the covariant derivative $D/Dt.$
The notion of a {\bfi horizontal curve} $v(t)$  on $V$ is
defined by the condition  that its covariant derivative vanishes.  A vector
tangent to $V$ is called {\bfi horizontal} if it is tangent to a
horizontal  curve. Correspondingly, the {\bfi horizontal space} at a point $v \in V$ is the space of all horizontal vectors at $v$.

In the case of an associated bundle we recall from \cite{CMR01a} the following formula that gives the relation between the covariant
derivative of the affine connection and the principal connection:
\[ \frac{D[q(t), m(t)]_G}{Dt} = \left[q(t), - \rho ^\prime\left(A\left(q(t),
\dot{q}(t)\right)\right)m(t)  + \dot{m}(t)\right]_G.
\]



This gives an affine connection on $Q
\times_G M$, called $\tilde{\nabla}^A$.
More precisely, let $\varphi : Q/G \rightarrow  Q \times_G M$ be a section of the
associated bundle and let $X(x) \in T_x(Q/G)$ be a given vector
tangent to $Q/G$ at $x$. Let $x(t)$ be a curve in $Q/G$ such that
$\dot{x}(0) = X(x)$; thus, $\varphi \left(x(t)\right)$ is a curve
in  $Q \times_G M$.  The covariant derivative of the section
$\varphi$ with respect to $X$ at $x$ is then, by  definition,
\begin{equation}
\label{indecedconnection} \tilde{\nabla}^A_{X(x)} \varphi  =
\left. \frac{D\varphi \left(x(t)\right)} {Dt} \right |_{t =
0}.
\end{equation}
Notice that we only need to know $\varphi$ along the curve $x(t)$ in order to
calculate  the covariant derivative.


\begin{dfn}
\label{dfn214} The associated bundle with standard fiber $\mathfrak{g}$, where
the action of $G$ on $\mathfrak{g}$ is the adjoint action, is called the {\bfi
adjoint bundle}, and is sometimes denoted $\operatorname{Ad}(Q)$. We will use
the notation $\tilde{\mathfrak{g}} : = \operatorname{Ad}(Q)$ in this paper. We
let $\tilde {\pi}_G : \tilde{\mathfrak{g}} \rightarrow Q/G$ denote the
projection  given by $\tilde{\pi}_G\left([q, \xi]_G\right) = [q]_G$.
\end{dfn}

The following properties hold.\\

(a)
Let $[q(s), \xi(s)]_G$ be any curve in $\tilde{\mathfrak{g}}$.
Then
\[ \frac{D[q(s), \xi(s)]_G}{Ds} = \left[q(s), - \left[A\left(q(s),
\dot{q}(s)\right), \xi(s)\right]  + \dot{\xi}(s)\right]_G.
\]

(b)
\label{lem216} Each fiber $\tilde{\mathfrak{g}}_x$ of $\tilde{\mathfrak{g}}$
carries a natural Lie algebra structure defined by
\[
\left[[q, \xi]_G, [q,\eta]_G\right] = \left[q,[\xi, \eta]\right]_G.
\]



Let $\pi :Q \rightarrow Q/G$ be a principal
bundle with structure group $G$, as before. The tangent lift of
the action of $G$ on $Q$ defines an action of $G$ on $TQ$ and so we
can form the quotient $(TQ)/G =: TQ / G$. There is a well defined
map $\tau _Q/G : TQ/G \rightarrow Q/G$ induced by the tangent of
the projection map $\pi: Q \rightarrow Q/G $ and given by
$[v_q]_G \mapsto [q ]_G$. The vector bundle structure of $TQ$ is inherited by
this bundle.\\

(c)
\label{lem221} The rules
\[ [v_q]_G + [u_q]_G = [v_q + u_q]_G \quad \mbox{and} \quad \lambda [v_q]_G =
[\lambda v_q]_G,
\] where $\lambda \in \mathbb{R}$, $v_q, u_q \in T _q Q$, and $[v_q]_G$ and
$[u_q]_G$ are their equivalence classes in the quotient $TQ/G $, define a
vector bundle structure on $TQ/G$ having base $Q/G$. The fiber $(TQ/G)_x$ is
isomorphic, as a vector space, to $T_qQ$, for each $x = [q]_G$.\\

(d)
\label{lem222} The map $\alpha_A : TQ/G \rightarrow T(Q/G) \oplus
\tilde{\mathfrak{g}}$ defined by
\[ \alpha _A\left([q, \dot{q}]_G\right) = T\pi (q, \dot{q}) \oplus [q, A(q,
\dot{q})]_G
\] is a well defined vector bundle isomorphism.  The inverse of $\alpha_A$ is
given by
\[ \alpha_A^{-1}\left((x, \dot{x}) \oplus [q, \xi]_G\right) = [ (x,
\dot{x})^h_q + \xi q]_G.
\]
\vspace{0.5cm}

An {\bfi action} $\rho : G\times V \rightarrow V$ of a Lie group
$G$ on a vector bundle $\tau  : V \rightarrow Q$ with extra structure $[\,, ]$ (a Lie algebra structure on each fiber of $V,$ in such a way that $V$ is a Lie algebra bundle),
$\omega$ (a $V$-valued 2-form on $Q$) and $\nabla$ (a covariant derivative $D/Dt$ for curves in $V$ related in the standard way to a connection $\nabla$ on $V$) , is a vector bundle
action such that, for each $g \in G$, $\rho _g : V \rightarrow V$ is a morphism that commutes with the structures given by $[\,,],$ $\omega,$ $\nabla$ in the vector bundles.

\begin{dfn}
\label{dfn519}
Let $\tau : V \rightarrow Q$ be a vector bundle and let $D/Dt$ be
the covariant derivative along curves  associated to a connection
$\nabla$ on $V$. Let  $\rho : G\times V \rightarrow V$ be a vector
bundle action covering the action $\rho _0 : G \times Q
\rightarrow Q$ which we assume that it is a principal bundle. Let $A$ be a principal connection on the principal
$G$-bundle $Q \rightarrow Q/G$. Let $v(t)$ be any curve in $V$ and let $q(t) = \tau
\left(v(t)\right)$ for all $t$. Choose $t_0$ and
let $q_0 = q(t_0)$. Let $g_q(t)$ and $q_h(t)$ be such that $q_h(t)$ is a
horizontal curve, $q(t) = g_q(t)q_h(t)$, and $g_q(t_0) = e$. Then
we define
\[
v_h(t) = g^{-1}_q(t)v(t),
\]
\[
\left.\frac{D^{(A,H)} v(t)}{Dt} \right |_{t = t_0} = \left.\frac{D
v_h(t)}{Dt} \right |_{t = t_0},
\]
and
\[
\left.\frac{D^{(A,V)} v(t)}{Dt} \right |_{t = t_0} = \left.\frac{D
v(t)}{Dt}
\right |_{t = t_0}
- \left.\frac{D^{(A,H)} v(t)}{Dt} \right |_{t = t_0}.
\]
We will call
\[
\frac{D^{(A,H)} v(t)}{Dt}
\]
the {\bfi $A$-horizontal covariant derivative} of $v(t)$ and
\[
\frac{D^{(A,V)} v(t)}{Dt}
\]
the {\bfi $A$-vertical covariant derivative} of $v(t)$. For $X\in
\mathfrak{X}^\infty(Q)$ and $v \in \Gamma(V)$ we also define $\nabla ^{(A,
H)}_Xv$ and $\nabla ^{(A, V)}_Xv$ by
\[
\nabla ^{(A, H)}_Xv(q_0) = \left. \frac{D^{(A, H)}}{Dt}v(t)\right
|_{t = t_0}
\]
and
\[
\nabla ^{(A, V)}_Xv(q_0) = \left. \frac{D^{(A, V)}}{Dt}v(t)\right
|_{t = t_0},
\]
where the covariant derivatives on the right hand side are taken
along any smooth curve  $q(t)$ in $Q$ satisfying $q(t_0) = q_0$,
$\dot{q}(t_0) = X(q_0)$, and $v(t) = v\left(q(t)\right)$ for all
$t$. We will call $\nabla ^{(A, H)}$ the {\bfi $A$-horizontal
component} and $\nabla ^{(A, V)}$ the {\bfi
$A$-vertical component} of the connection $\nabla$.
\end{dfn}

We see from this definitions that
\[
\nabla _Xv(q_0) = \nabla ^{(A, H)}_Xv(q_0) + \nabla ^{(A,
V)}_Xv(q_0).
\]

The following lemma gives, in particular,  an alternative
characterization of $\nabla ^{(A, H)}_Xv$ and $\nabla ^{(A, V)}_Xv$
for a $G$-invariant section $v$ of $V$. It also shows that, when
restricted to invariant sections $v \in \Gamma ^G (V )$ and invariant functions on $Q$, the
operator $\nabla ^{(A,H)} $ has the formal pro\-per\-ties of a
connection.\\

If $G\times V \rightarrow V$ is a vector bundle action on $\tau : V \rightarrow
Q$, a section $s : Q \rightarrow V$ is called an {\bfi invariant section} if
for all $g \in G$  and all $q \in Q$ we have $gs(q) = s(gq)$. The set
$\Gamma^G(V)$ of invariant sections of $V$ is a subspace of $\Gamma(V)$.

\begin{lem}
\label{lem5110} \begin{itemize}

\item[{\rm (a)}] Let $\operatorname{Hor}^A(X) \equiv X_H$ and
$\operatorname{Ver}^A (X) \equiv X_V$, the horizontal and the vertical components of the vector field $X$, respectively.
Then we have, for each $q_0
\in Q$, each $X \in \mathfrak{X}^{\infty}(Q)$
and each $G$-invariant section $v \in \Gamma ^G(V)$,
\[
\nabla ^{(A, H)}_Xv(q_0) = \nabla _{X_ H}v(q_0)
\]

and

\[
\nabla ^{(A, V)}_Xv(q_0) = \nabla _{X_ V}v(q_0).
\]

\item[{\rm (b)}]
Let $v \in \Gamma^G(V)$ and let $q(t)$ be any curve in $Q$
such that $\dot{q}(t_0) = X(q_0)$. Define $q_h(t)$ and $g_q(t)$
as in
Definition {\rm \ref{dfn519}}.
Then
\[
\left.\nabla ^{(A, H)}_Xv(q_0) = \frac{D}{Dt}g^{-1}_q(t)v(t)
\right |_{t = t_0} = \left. \frac{D}{Dt} v (q _h (t)) \right |_{t
= t_0}
\]
and
\[
\left.\nabla ^{(A, V)}_Xv(q_0) = \frac{D}{Dt}g_q(t)v(t_0) \right
|_{t = t_0}.
\]
In particular, $\nabla ^{(A, V)}_Xv(q_0)$ depends only on $\xi _0 =
\dot{g}_q(t_0)$ and $v(q_0) = v(t_0)$.

\item[{\rm (c)}]
Let $v \in \Gamma ^G(V)$, $q(t)$ a curve in $Q$, and $v(t) = v\left(q(t)\right)$ for all $t$. Then
\[
\left. \frac{D^{(A, H)}}{Dt}v(t)\right |_{t = t_0} =
\left.\frac{D}{Dt}g^{-1}_q(t)v(t) \right |_{t = t_0} = \left. \frac{D}{Dt} v (q
_h (t)) \right |_{t = t_0}
\]
and
\[
\left. \frac{D^{(A, V)}}{Dt}v(t)\right |_{t = t_0} =
\left.\frac{D}{Dt}g_q(t)v(t_0) \right |_{t = t_0}.
\]
\end{itemize}
\end{lem}

Let $\tau: V \rightarrow Q$ be a given vector bundle and let $\rho :  G\times V
\rightarrow V$ denote a given $G$-action on $V.$ Then the quotient $V/G$ carries a naturally defined vector bundle
structure over the base $Q/G$,  say $\tau /G : V/G \rightarrow Q/G$, where
$(\tau /G) ([v]_G)$ is defined by $(\tau /G) ([v]_G) = [\tau v]_G$.
The projection $\pi _G(V): V \rightarrow V/G$ is a surjective vector bundle
homomorphism covering $\pi$, and the restriction $\pi _G(V)|\tau ^{-1}(q): \tau
^{-1}(q) \rightarrow (\tau /G) ^{-1}([q]_G)$ is a linear isomorphism for each
$q \in Q$.
Let $G\times V \rightarrow V$ be a vector bundle action and let
$\pi _G(V): V \rightarrow V/G$ be the vector bundle homomorphism described above.
Then $ ( \pi _G(V) ) ^\ast : \Gamma(V/G) \rightarrow
\Gamma^G(V)$ is a linear isomorphism.

\begin{dfn}
\label{dfn5120} Let $[q_0, \xi _0]_G \in \tilde{\mathfrak{g}}$ with
$\tilde{\pi}_G [q_0, \xi _0]_G = [q_0]_G$ and let $[v]_G \in \Gamma(V/G)$,
where $v \in \Gamma ^G(V).$
Let $I_G^V(TQ)\rightarrow Q/G$ be the \textbf{vertical invariant bundle} whose fiber is $I_G^V(TQ)_x:=\{Y\in\mathfrak{X}^{\infty}(\pi^{-1}(x))|g^*Y=Y\}.$

Consider the horizontal invariant bundle, that is, the vector space of all horizontal invariant vector fields on $Q$ along $\pi^{-1}(x)$ given by $I_G^H(TQ)_x:=\{X:\pi^{-1}(x)\rightarrow TQ:X(q)\in \text{Hor}T_qQ,\ \forall q\in Q,\ g^*X=X\}.$

Let $Y_0 \in
I^V_{G}(TQ)$ such that $\beta _A(Y_0) = [q_0, \xi _0]_G$, where the map
\[ \beta _A : I^V_{G}(TQ) \rightarrow \tilde{\mathfrak{g}}
\] is the well defined Lie algebra isomorphism given by
\[ \beta _A(Y) = [q, A\left(Y(q)\right)]_G,
\] where $Y \in I^V_{G}(TQ)_x$, $x \in Q/G $, and $q\in\pi ^{-1}(x)$ is
arbitrary.

Then \begin{itemize}

\item[{\rm (a)}] The {\bfi quotient, or reduced, vertical connection} is
defined by
\[ \left[\nabla ^{(A, V)}\right] _{G ,[q_0, \xi _0]_G}[v]_G = \left[\nabla
^{(A, V)}_{Y_0}v\right]_G.
\]
\item[{\rm (b)}] Let  $\bar{X}_0 \in T_{[q_0]_G}(Q/G)$, define $X_0 = \bar{X}_0
^{h}$, so $X_0 \in I^H_{G}(Q)$, and $Z_0 = X_0 + Y_0$. The {\bfi quotient, or
reduced, connection} is defined by the condition
\[ \left[\nabla ^{(A)}\right] _{G ,\bar{X}_0 \oplus [q_0, \xi _0]_G}[v]_G =
\left[ \nabla _{Z_0}v \right] _G,
\] or by the equivalent condition
\[      \left[\nabla ^{(A)}\right] _{G ,\bar{X}_0 \oplus [q_0, \xi _0]_G}[v]_G
=      \left[\nabla ^{(A, H)}\right] _{G ,\bar{X}_0}[v]_G +      \left[\nabla
^{(A, V)}\right] _{G ,[q_0, \xi _0]_G}[v]_G.
\]
\end{itemize}
\end{dfn}

The following lemma establishes the link between the notions of quotient
vertical co\-va\-riant derivative (resp. quotient co\-va\-riant derivative) and
quotient vertical connection (resp. quotient connection).
\begin{lem}
\label{lem5121} Let $[q_0, \xi _0]_G \in \tilde{\mathfrak{g}}$ and let $[v]_G
\in \Gamma(V/G)$, where $v \in \Gamma ^G(V)$ as above. Let $[q]_G(t) = x(t)$ be any curve in $Q/G$ such that
$\tilde{\pi}_G\left([q_0, \xi _0]_G \right) = [q]_G(t_0)$ and let, with a
convenient abuse of notation, $[v]_G(t) = [v]_G\left([q]_G(t)\right)$. Assume
that $D/Dt$ is $G$-invariant. Then we have 

\[ \left[\left.\frac{D^{(A, V)}}{Dt}\right]_{G,[q_0, \xi _0]_G} [v]_G(t) \right
|_{t = t_0} = \left[\left.\frac{D g_0(t)v_0}{Dt} \right |_{t = t_0}\right]_G= \left[\nabla ^{(A, V)}\right] _{G ,[q_0, \xi _0]_G}[v]_G
\left([q]_G(t_0)\right),
\]

where $v_0 = v(q_0)$ and $g_0(t)$ is any curve on $G$ such that  $g_0(t_0) = e$
and  $\dot{g}_0(t_0) = \xi_0$.
\end{lem}

\subsection{The Lagrange-d'Alembert-Poincar{\'e} equations}\label{appendixb}

In this section we briefly review several definitions and results, stated and proved in  \cite{CMR01b}.

\paragraph{The Lagrange-d'Alembert principle.}
Let $Q$ be a
manifold and let $\mathcal{D}$ be a distribution on $Q$, that is,
$\mathcal{D}$ is a subbundle of $TQ$ and  let $L : Q \rightarrow
\mathbb{R}$ be a given Lagrangian. We can state the
Lagrange-d'Alembert principle as follows: \vspace{12pt}

{\it A curve $q(t)$, $t \in [t_0, t_1]$, on $Q$ is an actual motion
of the system if and only if $\dot{q}(t) \in \mathcal{D}_{q(t)}$ for
all $t$ and, besides, for any deformation $q(t, \lambda)$  of $q(t)$
such that the corresponding variation
$$
\delta q(t) = \frac{\partial q(t, \lambda)}{\partial \lambda}\Big |
_{\lambda = 0}
$$
satisfies $\delta q(t) \in \mathcal{D}_{q(t)}$ for all $t$, the
following condition holds
$$
\delta \int_{t_0}^{t_1}L(q, \dot{q})dt = 0.
$$
}

Here we use the standard notation in the calculus of variations,
namely,
$$
\delta \int_{t_0}^{t_1}L(q, \dot{q})dt = \frac{\partial}{\partial
\lambda}\Big | _{\lambda = 0} \int_{t_0}^{t_1} L\left(q(t, \lambda),
\dot{q}(t, \lambda)\right)dt.
$$
Using the Lagrange-d'Alembert Principle, equations of motion, called\textit{ Lagrange-d'Alembert equations}, can be written in a local chart as follows
\begin{align}
\frac{\partial L}{\partial q}(q, \dot{q})
-
\frac{d}{dt}\frac{\partial L}{\partial \dot{q}}(q, \dot{q})
&
\in( \mathcal{D}_q)^0,\\
(q, \dot{q})
&
\in  \mathcal{D}_q.
\end{align}

The \textit{form} of these equations is \textit{independent}
of the choice of coordinates, which is one of the great advantages of Lagrange-d'Alembert Principle.

\paragraph{Nonholonomic Systems with Symmetry.} \label{Nonholonomic Systems with Symmetry}
Now we shall assume
that $\pi : Q \rightarrow Q/G$ is a principal bundle with structure
group $G$ and we denote by $\mathcal{V}$ the vertical distribution,
that is, $\mathcal{V}_q = T_q\left(\pi ^{-1}([q]_G)\right)$, for
each $q \in Q$, which is obviously an integrable distribution whose
integral manifolds are the group orbits $\pi ^{-1}([q]_G)$.

Following \cite{BKMM96} let us consider the following condition, for simplicity:
\vspace{12pt}

(A1) DIMENSION ASSUMPTION. {\it For each $q \in Q$ we have the
equality

\begin{equation}\label{hipdimension}
T_q Q = \mathcal{D} _q + \mathcal{V}_q.
\end{equation}}

It is easy to see that, under assumption (A1), the dimension
of the space $\mathcal{S} _q = \mathcal{D} _q \cap \mathcal{V} _q$
does not depend on $q \in Q$, and moreover, the collection of spaces
$\mathcal{S} _q$, $q \in Q$, is a subbundle of $\mathcal{D}$, of
$\mathcal{V},$ and of $TQ$.

Also we consider the following condition: \vspace{12pt}

(A2) INVARIANCE OF $\mathcal{D}$. {\it The distribution
$\mathcal{D}$ is $G$-invariant.} \vspace{12pt}

It follows immediately from (A1) and (A2) that
$\mathcal{S}$ is a $G$-invariant distribution.

\paragraph{The Nonholonomic Connection.} It is known, and easy to
prove, that there is always a $G$-invariant metric on $Q$. See, for
example, \cite{CMR01a}.
In many important physical examples there is a natural way of
choosing such a metric, representing, for instance, the
inertia tensor of the system (see \cite{BKMM96}).

Let us choose an invariant metric on $Q.$
Then, under
condition (A1), we can define uniquely a principal connection
form $A: TQ \rightarrow \mathfrak{g}$ such that the horizontal
distribution $\operatorname{Hor}^A TQ$ satisfies the condition that,
for each $q$, the space $\operatorname{Hor}^A T_q Q$ coincides with
the orhogonal complement $\mathcal{H}_q$ of the space
$\mathcal{S}_q$ in $\mathcal{D}_q$.
This connection is called the {\bfi nonholonomic connection.}

For
each $q \in Q$, let us denote $\mathcal{U}_q$ the orthogonal
complement of $\mathcal{S}_q$ in $\mathcal{V}_q$. Then it is easy to
see that $\mathcal{U}$ is a distribution and we have the Whitney sum
decomposition
$$
TQ = \mathcal{H}\oplus \mathcal{S}\oplus \mathcal{U}.
$$
We obviously have $\mathcal{D} = \mathcal{H}\oplus \mathcal{S}$ and $\mathcal{V} = \mathcal{S}\oplus \mathcal{U}.$

Under the additional assumption (A2), all three distributions
$\mathcal{H}$, $\mathcal{S}$ and $\mathcal{U}$ are $G$-invariant,
therefore we can write,
$$
TQ/G = \mathcal{H}/G\oplus \mathcal{S}/G\oplus \mathcal{U}/G.
$$

\paragraph{The Geometry of the Reduced Bundles.}\label{The Geometry of the Reduced Bundles} Recall from
\cite{CMR01a} that there is a vector bundle isomorphism
$$
\alpha _A : TQ/G \rightarrow T(Q/G)\oplus \tilde{\mathfrak{g}},
$$
where $\tilde{\mathfrak{g}}$ is the adjoint bundle of the principal
bundle $Q$ defined as follows
$$
\alpha _A [q, \dot{q}]_G = T\pi (q, \dot{q})\oplus [q, A(q,
\dot{q})]_G.
$$
Notice that the bundle $T(Q/G)\oplus \tilde{\mathfrak{g}}$ does not
depend on the connection $A$, however the vector bundle isomorphism
$\alpha_A$ does depend on $A$. It is easy to see that $\alpha_A(\mathcal{H}/G) = T(Q/G),$ and
$\alpha_A(\mathcal{V}/G) = \tilde{\mathfrak{g}}.$
Define the subbundles $\tilde{\mathfrak{s}}$ and
$\tilde{\mathfrak{u}}$ of  $\tilde{\mathfrak{g}}$ by $\tilde{\mathfrak{s}} = \alpha_A(\mathcal{S}/G)$
and $\tilde{\mathfrak{u}} = \alpha_A(\mathcal{U}/G)$
respectively. Clearly, we have,
$$
\tilde{\mathfrak{g}} = \tilde{\mathfrak{s}}\oplus
\tilde{\mathfrak{u}}.
$$


\paragraph{Lagrange-d'Alembert-Poincar\'{e} Equations.}\label{reducedLPE}

Let
$l : T(Q/G)\oplus \tilde{\mathfrak{g}} \rightarrow \mathbb{R}$ be a given reduced Lagrangian. We will represent an element of
$T(Q/G)\oplus \tilde{\mathfrak{g}}$ using the notation
$(x, \dot{x}, \bar{v})$,
where
$x \in Q/G,$
$(x,\dot{x}) \in T_x(Q/G)$
and
$\bar{v} \in \tilde{\mathfrak{g}}_x.$
Care should be exercised using this notation since the bundle
$T(Q/G)\oplus \tilde{\mathfrak{g}}$
is not necessarily trivial.
However, we can still give a precise meaning to the
partial derivatives
\[
\frac{\partial l}{\partial x}, \quad \frac{\partial l}{\partial
\dot{x}} \quad  \mbox{and} \quad \frac{\partial l}{\partial \bar{v}},
\]
upon the choice of an affine connection
$\nabla$ on $Q/G$ and an affine connection
$\tilde{\nabla}^A$
on the vector bundle
$\tilde{\mathfrak{g}},$ which naturaly gives an affine connection on
$T(Q/G)\oplus \tilde{\mathfrak{g}},$
following \cite{CMR01a}.
First of all,  since $\tilde{\mathfrak{g}}$ and $T(Q/G)$ are vector bundles we
may interpret  the last two derivatives as being fiber derivatives
which are elements of the dual bundles
$T^{\ast}(Q/G)$ and $\tilde{\mathfrak{g}}^{\ast}$,  respectively,
and  for this, one does not need any extra structure other than the bundle structure.
In other words, for given $(x_0, \dot{x}_0, \bar{v}_0)$ and $(x_0,
x^\prime, \bar{v}^\prime)$ we define
\[
\left.\frac{\partial l}{\partial \dot{x}}(x_0, \dot{x}_0, \bar{v}_0)
\cdot x^\prime = \frac{d}{ds}\right |_{s = 0}l(x_0, \dot{x}_0 + s
x^\prime, \bar{v}_0)
\]
and
\[
\left.\frac{\partial l}{\partial \bar{v}}(x_0, \dot{x}_0, \bar{v}_0)
\cdot \bar{v}^\prime = \frac{d}{ds}\right |_{s = 0}l(x_0, \dot{x}_0
, \bar{v}_0 + s \bar{v}^\prime).
\]
Now we shall define the derivative
$\partial l /
\partial x $. Let $(x_0,
\dot{x}_0,\bar{v}_0)$ be  a given element of $T(Q/G)\oplus
\tilde{\mathfrak{g}}$.  For any given curve $x(s)$ on $Q/G$, let
$\left(x(s), \bar{v}(s)\right)$ be the horizontal  lift of $x(s)$
with respect to the connection $\tilde{\nabla}^A $ on
$\tilde{\mathfrak{g}}$ such that $\left(x(0), \bar{v}(0)\right) =
\left(x_0, \bar{v}_0\right),$ and let $\left(x(s), u(s)\right)$ be
the horizontal lift of $x(s)$ with respect to the connection
$\nabla$ such that $\left(x(0), u(0)\right) =
\left(x_0,\dot{x}_0\right)$. (Notice that in general, $\left(x(s),
u(s)\right)$ is {\it not} the tangent vector $\left(x(s),
\dot{x}(s)\right)$ to $x(s)$). Thus, $\left(x(s), u(s),
\bar{v}(s)\right)$ is a horizontal curve with respect to the
connection $C = \nabla \oplus \tilde{\nabla}^A$ naturally defined on
$T(Q/G)\oplus \tilde{\mathfrak{g}}$ in terms of the connection
$\nabla$ on $T(Q/G)$ and the connection $\tilde{\nabla}^A$ on
$\tilde{\mathfrak{g}}$.
\begin{dfn}\label{dfnloc324}
The {\bfi covariant derivative} of $l$ with respect to $x$ at $(x_0,
\dot{x}_0, \bar{v}_0)$ in the direction of $\left(x(0),
\dot{x}(0)\right)$ is defined by
\[
\left.\frac{\partial^C l}{\partial x}(x_0, \dot{x}_0,
\bar{v}_0)\left(x(0), \dot{x}(0)\right) = \frac{d}{ds}\right |_{s =
0} l\left(x(s), u(s), \bar{v}(s)\right).
\]
We shall often write
\[
\frac{\partial^C l}{\partial x} \equiv \frac{\partial l}{\partial
x},
\]
whenever there is no danger of confusion.
\end{dfn}

The covariant derivative on a given vector bundle, for instance
$\tilde{\mathfrak{g}}$, induces a corresponding covariant derivative
on the dual bundle, in our case $\tilde{\mathfrak{g}}^{\ast}$.  More
precisely, let $\alpha (t) $ be a curve in
$\tilde{\mathfrak{g}}^{\ast}$.  We define the covariant derivative
of $\alpha (t)$ in such a way that for  any curve $\bar{v}(t)$ on
$\tilde{\mathfrak{g}}$  such that both  $\alpha (t)$ and
$\bar{v}(t)$ project on the same curve $x(t)$ on $Q/G$,  we have
\[
\frac{d}{dt}\left\langle\alpha (t), \bar{v}(t)\right\rangle =
\left\langle\frac{D \alpha (t)}{Dt}, \bar{v}(t)\right\rangle +
\left\langle\alpha (t), \frac{D\bar{v}(t)}{Dt}\right\rangle .
\]
Likewise we can define the covariant derivative in the vector bundle
$T^\ast(Q/G)$. Then we obtain a covariant derivative on the vector
bundle $T^\ast(Q/G)\oplus\tilde{\mathfrak{g}}^\ast.$

It is in the sense of this definition that terms like
\[
\frac{D}{Dt}\frac{\partial l}{\partial \dot{x}}
\]
and
\[
\frac{D}{Dt}\frac{\partial l}{\partial\bar{v}}
\]
should be interpreted. In the first case,  $D/Dt$ means the covariant derivative in the bundle
$T^{\ast}(Q/G)$ and, in the second case, $D/Dt$  is the covariant
derivative in the bundle $\tilde{\mathfrak{g}}^{\ast}$.


Recall that the curvature $B\equiv B^{\mathcal{A}}$ of $\mathcal{A}$ is given by
\begin{eqnarray}\label{curvatura}
B(q)(X_1(q),X_2(q))=-\mathcal{A}([X_1^h,X_2^h])(q),
\end{eqnarray}
where $X,Y$ are vector fields on $Q/G$ and $X_1^h, X_2^h$ are their horizontal lifts, respectively.

The curvature 2-form $B\equiv B^{\mathcal{A}}$ of the connection $\mathcal{A}$ induces a $\widetilde{\mathfrak{g}}$-valued 2-form $\widetilde{B}\equiv\widetilde{B}^{\mathcal{A}}$ on $Q/G$ given by
\begin{equation}\label{betilde}
\widetilde{B}(x)(\delta x, \dot{x})=[q,B(q)(\delta q,\dot{q})]_G,
\end{equation}
where for each $(x,\dot{x})$ and $(x,\delta x)$ in $T_x(Q/G),$ $(q,\dot{q})$ and $(q,\delta q)$ are any elements of $T_qQ$ such that $\pi(q)=x,$ $T\pi(q,\dot{q})=(x,\dot{x})$ and $T\pi(q,\delta q)=(x,\delta x)$ (see \cite{CMR01a} for the proof).

The $\widetilde{\mathfrak{g}}$-valued 2-form $\widetilde{B}$ on $Q/G$ will be called the \textbf{\emph{reduced curvature form}}.\label{def:Btilde}


\paragraph{The La\-gran\-ge-d'Al\-em\-bert-Poincar\'{e} Equations.} In \cite{CMR01b} and \cite{CMR01a} the theory of Lagrange-Poincar\'e and Lagrange-d'Alembert-Poincar\'e equations has been developed. Here we recall the results that we use.

\begin{thm}\label{thm326}
Let $q (t)$ be a curve in $Q$ such that $( q (t), \dot{q} (t) ) \in
{\mathcal D}_{q (t) }$ for all $t$ and let $( x (t),  \dot{x}(t),
\bar{v}(t) ) = \alpha _A \left( [ q (t), \dot{q} (t) ]_{G}  \right)$
be the corresponding curve in $T (Q/G) \oplus \tilde{\mathfrak{s}
}$.  The following conditions are equivalent:
\begin{itemize}
\item[{\rm (i)}]
The {\bfi La\-gran\-ge-d'Al\-em\-bert principle} holds:
\[
\delta \int_{t_0}^{t_1}L(q,\dot{q})dt = 0
\]
for variations $\delta q$ of the curve $q$ such that $\delta q(t_i)
= 0$, for $i = 0, 1,$ and $\delta q (t) \in {\mathcal D} _{q (t)}$
for all $t$.
\item[{\rm (ii)}]
The {\bfi reduced La\-gran\-ge-d'Al\-em\-bert principle} holds: The
curve $x(t)\oplus \bar{v}(t)$ satisfies
\[
\delta \int_{t_0}^{t_1} l\left(x(t), \dot{x}(t), \bar{v}(t)\right)dt
= 0,
\]
for variations $\delta x \oplus \delta ^A\bar{v}$ of the curve
$x(t)\oplus \bar{v}(t)$, where $\delta ^A\bar{v}$ has the form
\[
\delta ^A\bar{v} = \frac{D \bar{\eta}}{Dt} + \left[\bar{v},
\bar{\eta}\right] + \tilde{B}(\delta x, \dot{x}),
\]
with the boundary conditions $\delta x(t_i) = 0$ and
$\bar{\eta}(t_i) = 0$, for $i = 0, 1,$ and where $\bar{\eta} (t) \in
\tilde{\mathfrak{s} } _{x (t) }$.
\item[{\rm (iii)}]
The following {\bfi vertical
La\-gran\-ge-d'Al\-em\-bert-Poincar\'{e} equations}, corresponding
to vertical variations,  hold:
\[
\frac{D}{Dt}\left.\frac{\partial l} {\partial \bar{v}}(x, \dot{x},
\bar{v})\right |_{\tilde{\mathfrak{s}}_x
} = \operatorname{ad}^{\ast}_{\bar {v}}\left.\frac{\partial
l}{\partial \bar{v}}(x, \dot{x}, \bar{v})\right |_{\tilde{\mathfrak{s}}_x},
\]
and the {\bfi horizontal La\-gran\-ge-d'Al\-em\-bert-Poincar\'{e}
equations}, corresponding to horizontal variations, hold:
\[
\frac{\partial^C l}{\partial x}(x, \dot{x}, \bar{v}) - \frac{D}{D
t}\frac{\partial l}{\partial \dot{x}}(x, \dot{x}, \bar{v}) =
\left\langle\frac{\partial l}{\partial \bar{v}} (x, \dot{x},
\bar{v}), \mathbf{i}_{\dot{x}}\tilde{B}(x) \right\rangle.
\]
\end{itemize}
In part (ii) of this theorem, if $\bar{v} = [q, v]_G$ with $v = A(q,
\dot{q})$ then $\bar{\eta}$ can be always written $\bar{\eta} = [q,
\eta]_G$, and the condition $\bar{\eta}(t_i) = 0$ for $i = 0, 1,$ is
equivalent to the condition $\eta(t_i) = 0$ for $i = 0, 1$. Also, if
$x(t) = [q]_G$ and $\bar{v} = \left[q, v\right]_G$ where $v =
A\left(q, \dot{q}\right)$, then variations $\delta x \oplus \delta
^A\bar{v}$ such that
\[
\delta ^A\bar{v} =\frac{D \bar{\eta}}{Dt} + \left[\bar{v},
\bar{\eta}\right] \equiv \frac{D[q, \eta]_G}{Dt} + \left[q,
[v,\eta]\right]_G
\]
with $\bar{\eta} (t_i) = 0$ (or, equivalently, $\eta (t_i) = 0$) for
$i = 0, 1,$ and $\bar{\eta} (t) \in \tilde{\mathfrak{s} }_{x (t) }$
correspond exactly to vertical variations $\delta q$ of the curve
$q$ such that $\delta q(t_i) = 0$ for $i = 0, 1$, and $\delta q (t)
\in \mathcal{S}_{q (t) }$, while variations $\delta x \oplus \delta
^A\bar{v}$ such that
\[
\delta ^A\bar{v} = \tilde{B}(\delta x, \dot{x})
\]
with $\delta x(t_i) = 0$ for $1 = 0, 1$, correspond exactly to
horizontal variations $\delta q$ of the curve $q$ such that $\delta
q(t_i) = 0.$

\end{thm}
\vspace{.3cm}
\textbf{Remarks.}

\vspace{.3cm}
\textbf{a)} Note that
 $\displaystyle\frac{D}{Dt}\left.\frac{\partial l}{\partial\overline{v}}(x,\dot{x},\overline{v})\right |_{\tilde{\mathfrak{s}}_x
}$ stands for $\displaystyle \left.\left(\frac{D}{Dt}\frac{\partial l}{\partial\overline{v}}(x,\dot{x},\overline{v})\right)\right |_{\tilde{\mathfrak{s}}_x
}.$

\vspace{.3cm}
\textbf{b)} In the case in which there is no restriction, then $\tilde{\mathfrak{s}}=\tilde{\mathfrak{g}}$, and the Lagrange-d'Alembert-Poincar\'e equations are exactly the Lagrange-Poincar\'e equations (see \cite{CMR01b,CMR01a}).

\bibliographystyle{plain}

\end{document}